\documentclass[11pt,reqno]{amsart}

\usepackage{amssymb, amsmath, amsthm}
\usepackage{hyperref}
\usepackage[alphabetic,lite]{amsrefs}
\usepackage{verbatim}
\usepackage{amscd}   
\usepackage[all]{xy} 
\usepackage{youngtab} 
\usepackage{young} 
\usepackage{ytableau}
\usepackage{tikz-cd}
\usepackage{cases}
\usepackage{array}
\usepackage{tabu}
\usepackage{calligra,mathrsfs}

\title{Equivariant $\D$-modules on alternating senary 3-tensors}

\author{Andr\'as C. L\H{o}rincz}
\address{Department of Mathematics, Purdue University, West Lafayette, IN 47907}
\email{alorincz@purdue.edu}

\author{Michael Perlman}
\address{Department of Mathematics, University of Notre Dame, Notre Dame, IN 46556}
\email{mperlman@nd.edu}


 \newtheorem{theorem}{Theorem}[section]
 \newtheorem{lemma}[theorem]{Lemma}
 \newtheorem{corollary}[theorem]{Corollary}
 \newtheorem{prop}[theorem]{Proposition}

\newtheorem*{class}{Classification of Simple Modules}
\newtheorem*{quiv}{Theorem on Quiver Structure}
\newtheorem*{loccoh}{Theorem on Local Cohomology}
\theoremstyle{remark}
 \newtheorem{remark}[theorem]{Remark}

\newcommand{\defi}[1]{{\upshape\sffamily #1}}
\DeclareMathOperator{\ShHom}{\mathscr{H}\text{\kern -3pt {\calligra\large om}}\,}

\newcommand{\C}{\mathbb{C}}
\newcommand{\Z}{\mathbb{Z}}
\renewcommand{\a}{\alpha}
\renewcommand{\b}{\beta}
\newcommand{\bo}{\bigoplus}
\newcommand{\g}{\gamma}
\renewcommand{\d}{\delta}
\newcommand{\D}{\mathcal{D}}

\newcommand{\bw}{\bigwedge}
\renewcommand{\det}{\textrm{det}}

\renewcommand{\ll}{\lambda}
\newcommand{\LL}{\Lambda}

\newcommand{\oo}{\otimes}

\renewcommand{\SS}{\mathbb{S}}
\newcommand{\sub}{\subset}

\newcommand{\F}{\mathcal{F}}
\renewcommand{\L}{\mathcal{L}}
\renewcommand{\O}{\mathcal{O}}

\newcommand{\R}{\mathcal{R}}
\newcommand{\Q}{\mathcal{Q}}

\newcommand{\Ext}{\operatorname{Ext}}

\newcommand{\GL}{\operatorname{GL}}
\newcommand{\gr}{\operatorname{gr}}
\newcommand{\Hom}{\operatorname{Hom}}
\newcommand{\End}{\operatorname{End}}

\newcommand{\rep}{\operatorname{rep}}

\newcommand{\SL}{\operatorname{SL}}
\newcommand{\sort}{\operatorname{sort}}

\newcommand{\Sub}{\operatorname{Sub}}
\newcommand{\Sym}{\operatorname{Sym}}

\newcommand{\Tot}{\operatorname{Tot}}

\newcommand\charC{{\operatorname{charC}}}

\renewcommand{\det}{\operatorname{det}}

\newcommand{\Gr}{\operatorname{Gr}}

\newcommand{\opmod}{\operatorname{mod}}

\newcommand{\bb}[1]{\mathbb{#1}}

\newcommand{\mf}[1]{\mathfrak{#1}}
\newcommand{\ol}[1]{\overline{#1}}

\newcommand{\la}{\langle}
\newcommand{\ra}{\rangle}

\def\lra{\longrightarrow}

\newcommand{\tn}{\textnormal}

\numberwithin{equation}{section}

\begin{document}

 \begin{abstract} 
We consider the space $X=\bw^3\bb{C}^6$ of alternating senary 3-tensors, equipped with the natural action of the group $\GL_6$ of invertible linear transformations of $\bb{C}^6$. We describe explicitly the category of $\GL_6$-equivariant coherent $\D_X$-modules as the category of representations of a quiver with relations, which has finite representation type. We give a construction of the six simple equivariant $\D_X$-modules and give formulas for the characters of their underlying $\GL_6$-structures.  We describe the (iterated) local cohomology groups with supports given by orbit closures, determining, in particular, the Lyubeznik numbers associated to the orbit closures. 
\end{abstract}

\maketitle

\section{Introduction}\label{sec:intro}

Let $V$ be a six-dimensional vector space, and let $X=\bw^3V$ be the space of alternating 3-tensors of $V$. This space has a natural action of $G=\GL(V)$ with five orbits. We describe these orbits below, choosing a basis $V=\la e_1,e_2,e_3,e_4,e_5,e_6\ra$.
\begin{itemize}
\item The zero orbit $O_0=\{0\}$.

\item The orbit $O_1$ of dimension $10$, with representative $e_1\wedge e_2\wedge e_3$, whose closure $\ol{O_1}$ is the affine cone over the Grassmannian $\Gr(3,V)\sub \bb{P}(\bw^3V)$ under the Pl\"{u}cker embedding.

\item The orbit $O_2$ of dimension $15$, with representative $e_1\wedge(e_2\wedge e_3+e_4\wedge e_5)$, whose closure $\ol{O_2}$ is the subspace variety
$$
\Sub_5(X)=\{T\in X \mid \exists Z\in \Gr(5,V), T\in \bw^3 Z\}.
$$

\item The orbit $O_3$ of dimension $19$, with representative $e_1\wedge(e_2\wedge e_3 +e_4\wedge e_5)+e_2\wedge e_4\wedge e_6$, whose closure $\ol{O_3}$ is the affine cone over the tangential variety to the Grassmannian $\Gr(3,V)$.

\item The dense orbit $O_4$ of dimension $20$, with representative $e_1\wedge e_2\wedge e_3+e_4\wedge e_5 \wedge e_6$.
\end{itemize}
\noindent Further, for each $k=0,1,2,3,4$, the Zariski closure of $O_k$ is the union of all $O_j$ with $j\leq k$.

The space $\bw^3 \C^6$ is an example in the subexceptional series of representations with finitely many orbits \cite[Section~6]{series}. Other representations from this series include the space of binary cubic forms and the space of $2\times 2\times 2$ hypermatrices. The equivariant $\D$-modules for these representations are described in \cite{bindmod} and \cite{mike}, respectively. For the irreducible representations that are spherical varieties, the categories of equivariant $\D$-modules have been described in \cite{catdmod}. The representations in the subexceptional series are some of the simplest representations that are not spherical. For the space of alternating senary 3-tensors, although the coordinate ring $\C[X]$ is not multiplicity-free as a $G$-module, the algebra of covariants $\C[X]^U$ is a polynomial ring \cite{brion}, where $U$ denotes a maximal unipotent subgroup of $G$.

By the Riemann--Hilbert correspondence, equivariant $\D$-modules correspond to equivariant perverse sheaves. The simple equivariant $\D$-modules are indexed by irreducible equivariant local systems on the orbits of the group action, but their explicit realization is in general difficult (see Open Problem 3 in \cite[Section~6]{mac-vil}).

Put $W=V^\ast$, let $S=\Sym(\bw^3 W)\cong \C[x_{i,j,k}\mid 1\leq i<j<k\leq 6]$ be the ring of polynomial functions on $X$, and $\D=\D_X=S\cdot \la \partial_{i,j,k}\mid 1\leq i<j<k\leq 6\ra$ be the Weyl algebra of differential operators on $X$ with polynomial coefficients. We write $\opmod_G(\D_X)$ for the category of $G$-equivariant coherent $\D$-modules on $X$.

\begin{class}
There are six simple $\GL_6(\C)$-equivariant $\D$-modules on $X=\bw^3 \C^6$. For all orbits $O\neq O_4$, there is a unique simple with support $\ol{O}$. These modules correspond to the trivial local systems on their respective orbits, and we denote them by $D_0=E$, $D_1$, $D_2$, and $D_3$. There are two simple equivariant $\D$-modules with full support: $S$ and $B_4$.

The holonomic duality functor fixes all of the simple modules. The Fourier transform swaps the modules in the two pairs $(S,E)$, $(B_4,D_1)$, and the other simples are fixed.
\end{class}

We give explicit constructions for all the simple equivariant $\D$-modules in Theorem \ref{thm:simples}. Using this, for each simple equivariant $\D$-module we determine the character of its underlying $G$-module structure in Section \ref{sec:character}. Similar formulas for characters were obtained in various other equivariant situations \cites{bindmod, claudiu1,claudiu5}.

The key object in the quiver structure of the category $\opmod_G(\D_X)$ in question is the semi-invariant $f$ of weight $(2^6)$ (see \cite[Proposition 5.7]{saki} and subsequent discussion). Note that $X\setminus O_4=\ol{O_3}$ and $\ol{O_3}=V(f)$. We will gain a lot of information about the category in question by studying the modules $S_f$ and $S_f\cdot \sqrt{f}$, where $S_f$ denotes localization of $S$ at the semi-invariant. The $\D$-module filtrations of these modules will tell us about which nontrivial extensions are possible. By (\cite{catdmod}, Proposition 4.9), these filtrations are dictated by the Bernstein-Sato polynomial (or $b$-function) of $f$. The $b$-function is (see \cite[Section 8]{kimu})
\begin{equation}\label{eq:roots}
b_f(s)=(s+1)(s+5/2)(s+7/2)(s+5).
\end{equation}

It is known that in the case of a group acting on a space with finitely many orbits, the category of equivariant $\D$-modules (or perverse sheaves) is equivalent to the category of finite-dimensional representations of a quiver with relations (see \cite{vilonen} and \cite{catdmod}). We now state the theorem on the quiver structure of the category of $G$-equivariant coherent $\D$-modules on $X$.

\begin{quiv}
There is an equivalence of categories
$$
\opmod_G(\D_X)\cong \rep(Q,I),
$$
where $\rep(Q,I)$ is the category of finite-dimensional representations of a quiver $Q$ with relations $I$. The quiver $Q$ is \smallskip

\[
\begin{tikzcd}
s \arrow[r, shift left, "\a_0"]
& d_3 \arrow[l, "\b_0"]
\arrow[r, shift left, "\a_1"]
& e \arrow[l, "\b_1"]\;\;\;\;\;\;\;\;\;
b_4 \arrow[r, shift left, "\g_0"]
& d_2 \arrow[r, shift left, "\g_1"]
\arrow[l, "\d_0"]
& d_1 \arrow[l, "\d_1"]\\
\end{tikzcd}
\]
The set of relations $I$ is given by all 2-cycles:
$$
\a_0\b_0=\b_0\a_0=\a_1\b_1=\b_1\a_1=\g_0\d_0=\d_0\g_0=\g_1\d_1=\d_1\g_1=0.
$$
\end{quiv}

Note that (each connected component of) this quiver appears also in \cite{catdmod}, and has finitely many (isomorphism classes of) indecomposable representations that can be described explicitly \cite[Theorem 2.13]{catdmod}.

For any $G$-stable closed subset $Z$ in $X$, the local cohomology modules $H^i_Z(S)$ are $G$-equivariant coherent $\D$-modules, for all $i\geq 0$. Explicit computations of local cohomology modules are in general notoriously difficult. Several results been obtained for several representations with finitely many orbits (see \cites{bindmod,lHorincz2018iterated,mike,claudiu2,claudiu4,claudiu3}). We state the theorem on their explicit $\D$-module structures in our case (in the following all short exact sequences are nonsplit).

\begin{loccoh}
The following are all the nonzero local cohomology modules of $S$ with support in an orbit closure:
\begin{itemize}
\item[(0)] $H^{20}_{O_0}(S) = E$. \smallskip
\item[(1)] $H^{10}_{\ol{O_1}}(S) = D_1, \;\; H^{13}_{\ol{O_1}}(S) = E, \;\; H^{15}_{\ol{O_1}}(S) = E$.\smallskip
\item[(2)] $0\to D_2 \to H^{5}_{\ol{O_2}}(S) \to D_1 \to 0, \; \; H^{7}_{\ol{O_2}}(S) = D_1, \; \; H^{10}_{\ol{O_2}}(S) = E.$\smallskip
\item[(3)] $0 \to D_3 \to H^1_{\ol{O_3}} (S) \to E \to 0.$
\end{itemize}
\end{loccoh} 

The most difficult part of the theorem above is proving part (2). For the proof, we use a desingularization of $\ol{O_3}$ and reduce the problem to the computation of the cohomology of some vector bundles on the Grassmannian. Computing these cohomology spaces is not an easy task, as these bundles are not semi-simple. Nevertheless, we are able to determine the relevant cohomological pieces that identify the respective simple equivariant $\D$-modules. For this identification, the explicit description of the equivariant simple $\D$-modules plays a crucial role. This technique could be a model for computing local cohomology in other representations with discriminants (see \cite{knopmenzel} and \cite{jerzydisc}).

The article is organized as follows. In Section \ref{sec:prelim} we introduce some basic terminology regarding representations of the general linear group, equivariant $\D$-modules, representation theory of quivers and the Borel-Weil-Bott theorem on the Grassmannian. In Section \ref{sec:cat} we give an explicit construction for all the simple equivariant $\D$-modules, and we determine the quiver corresponding to the category $\opmod_G(\D_X)$. In Section \ref{sec:character} we determine for each simple equivariant $\D$-module its character as a $G$-module. In Section \ref{sec:loccoh} we determine all the (iterated) local cohomology modules with supports given by orbit closures for all simple equivariant $\D$-modules, and describe the Lyubeznik numbers of the orbit closures.

\section{Preliminaries}\label{sec:prelim}
\subsection{Representations and Characters} Let $V$ be a six dimensional complex vector space and write $W=V^{\ast}$. Let $\LL$ be the set of isomorphism classes of finite dimensional irreducible representations of the general linear group $\GL=\GL(W)\cong \GL_6(\C)$. Then $\LL$ can be identified with the set of \defi{dominant} integral weights $(\ll_1\geq \ll_2\geq \cdots \geq \ll_6)\in \Z^6$. We often write $(a^6)$ for the weight $(a,a,a,a,a,a)$. We write $\SS_{\ll}V$ for the irreducible representation (or Schur functor) corresponding to $\ll\in \LL$ (for details, see \cite{jerzy}). 

For $\ll=(\ll_1\geq \ll_2\geq \cdots \geq \ll_6)\in \LL$, we denote its dual by $\ll^*=(-\ll_6,-\ll_5,-\ll_4,-\ll_3,-\ll_2,-\ll_1)$, so that $\SS_\ll W = \SS_{\ll^*} V$.

A $G$-representation $M$ is \defi{admissible} if it decomposes as a direct sum
$$
M=\bo_{\ll \in \LL} (\SS_{\ll}W)^{\oplus a_{\ll}}
$$
for $a_{\ll}\in \Z_{\geq 0}$. Consider the set $\Gamma(\GL)=\Z^{\LL}$ of maps $\LL\to \Z$. Then $\Gamma(\GL)$ is the \defi{Grothendieck group of virtual admissible representations}. Given an admissible representation $M$ as above, we write the formal sum
$$
[M]=\sum_{\ll\in \LL} a_{\ll}\cdot e^{\ll}
$$
for the corresponding element in $\Gamma(\GL)$, where $e^{\ll}$ corresponds to $S_{\ll}W$ and $a_{\ll}$ is the multiplicity. Using this notation, we write
$$
\la [M], e^{\ll}\ra = a_{\ll}
$$
for the multiplicity of $e^{\ll}$ in $[M]$ in $\Gamma(\GL)$. We also have a natural (partially defined) multiplication on $\Gamma(\GL)$ induced by 
\[e^\ll \cdot e^\mu = e^{\ll+\mu}, \,\,\, \mbox{ for } \ll, \mu \in \LL.\]
With this, we can make sense of inverting some virtual representations such as $(1-e^{\ll})$:
\[ \frac{1}{1-e^{\ll}} = 1 + e^{\ll} + e^{2\ll} + \cdots.\]

\subsection{Equivariant $\D$-modules}

Let $\mf{g}$ denote the Lie algebra of $G$. Then differentiating the $G$-action on $X$ we obtain a map from $\mf{g}$ to space of algebraic vector fields on $X$, which in turn yields a map $\mf{g} \to \D_X$. A $\D_X$-module $M$ is \defi{equivariant} if we have a $\D_{G\times X}$-isomorphism $\tau: p^*M \rightarrow m^*M$, where $p: G\times X\to X$ denotes the projection and $m: G\times X\to X$ the map defining the action, with $\tau$ satisfying the usual compatibility conditions (see \cite[Definition 11.5.2]{hot-tak-tan}). This amounts to $M$ admitting an algebraic $G$-action whose differential coincides with the $\mf{g}$-action induced by the natural map $\mf{g} \to \D_X$. The category $\opmod_G(\D_X)$ of equivariant $\D$-modules is a full subcategory of the category $\opmod(\D_X)$ of all coherent $\D$-modules, which is closed under taking submodules and quotients. For a $G$-stable closed subset $Z$ of $X$, we denote by $\opmod_G^Z(\D_X)$ the full subcategory of $\opmod_G(\D_X)$ consisting of equivariant $\D$-modules with support contained in $Z$.

Since $G$ acts on $X$ with finitely many orbits, every module in $\opmod_G(\D_X)$ is regular and holonomic \cite[Theorem~11.6.1]{hot-tak-tan}, and via the Riemann--Hilbert correspondence $\opmod_G(\D_X)$ is equivalent to the category of $G$-equivariant perverse sheaves on $X$. This category is equivalent to the category of finitely generated modules over a finite dimensional $\bb{C}$-algebra (see \cite[Theorem 4.3]{vilonen} or  \cite[Theorem 3.4]{catdmod}), or to the category of representations of a quiver with relations (see \cite{ass-sim-sko}). For more details on categories of equivariant $\D$-modules, cf. \cite{catdmod}.

One important construction of objects in $\opmod_G(\D_X)$ comes from considering local cohomology functors $H^i_Z(\bullet)$, for $Z=\ol{O_k}$ an orbit closure in $X$. Namely, for each $i\geq 0$ and each $M\in\opmod_G(\D_X)$ we consider $i$-th local cohomology module $H^i_Z(M)$ of $M$ with support in $Z$, which is an element of $\opmod_G^Z(\D_X)$. If we write $c = \dim(X) - \dim(Z)$ for the codimension of $Z$ in $X$ and take $M=S$, then $H^i_Z(S) = 0$ for all $i<c$, and $H^c_Z(S)$ contains the simple $D_k$ as a $\D$-submodule (with the convention $S=D_4$, $E=D_0$), and the quotient $H^c_Z(S) / D_k$ is an element in $\opmod_G^{\ol{O_k} \setminus O_k}(\D_X)$.

Another important construction of objects in $\opmod_G(\D_X)$ comes from considering the (twisted) Fourier transform \cite[Section 4.3]{catdmod}. This functor gives a self-equivalence
\[\F : \opmod_G(\D_X) \xrightarrow{\sim} \opmod_G(\D_X).\]
Moreover, since  $\det(\bw^3 W^\ast)=\SS_{(-10^6)}W$, for an object $M\in\opmod_G(\D_X)$, we have in $\Gamma(\GL)$
\begin{equation}\label{eq:fourier}
[\F(M)] = [M]^* \cdot  e^{(-10^6)}.
\end{equation}

Let $S=\Sym(\bw^3 W)\in \opmod_G(\D_X)$ be the ring of polynomial functions on $\bw^3 V$. The action of $\GL$ on $\bw^3 V$ extends to an action on $S$ and the character of $S$ is (see \cite{brion} or \cite[Section 6]{series}):
\begin{equation}\label{eq:charS}
[S]=\frac{1}{(1-e^{(1,1,1,0,0,0)})(1-e^{(2,1,1,1,1,0)})(1-e^{(2,2,2,1,1,1)})(1-e^{(2,2,2,2,2,2)})(1-e^{(3,3,2,2,1,1)})}
\end{equation}
Write $E=\F(S)$ for the simple $G$-equivariant $\D$-module whose support is equal to the origin $O_0$. By (\ref{eq:fourier}) we have:
\begin{equation}\label{eq:charE}
[E]=\frac{e^{(-10^6)}}{(1-e^{(0^3,-1^3)})(1-e^{(0,-1^4,-2)})(1-e^{(-1^3,-2^3)})(1-e^{(-2^6)})(1-e^{(-1^2,-2^2,-3^2)})}
\end{equation}

\subsection{Quivers}

We briefly review some basic material on the representation theory of quivers, following \cite{ass-sim-sko}.  A \defi{quiver} $Q$ is an oriented graph, i.e., a pair $Q=(Q_0,Q_1)$ formed by a finite set of vertices $Q_0$ and a finite set of arrows $Q_1$. An arrow $\alpha\in \mathcal{Q}_1$ has source $s(\alpha)\in \mathcal{Q}_0$ and a target $t(\alpha)\in \mathcal{Q}_0$. A directed path $p$ in $\mathcal{Q}$ from $x\in \mathcal{Q}_0$ to $y\in \mathcal{Q}_0$ is a sequence of arrows $\alpha_1,\cdots, \alpha_k$ such that $s(\alpha_1)=x$, $t(\alpha_k)=y$, and $s(\alpha_i)=t(\alpha_{i-1})$. In which case, we say $p$ has source $x$ and target $y$. A \defi{relation} in $\mathcal{Q}$ is a linear combination of paths of length at least two having the same source and target. We define a \defi{quiver (with relations)} $(Q,I)$ to be a quiver $Q$ together with a finite set of relations $I$.

A \defi{(finite-dimensional) representation} $M$ of a quiver $(Q,I)$ is a family of (finite-dimensional) vector spaces $\{M_x\,|\, x\in Q_0\}$ together with linear maps $\{M(\alpha) : M_{s(\alpha)}\to M_{t(\alpha)}\, | \, \alpha\in Q_1\}$ satisfying the relations induced by the elements of $I$. A morphism $\phi:M\to N$ of two representations $M,N$ of $(Q,I)$ is a collection of linear maps  $\phi = \{\phi(x) : M_x \to N_x\,| \,x\in Q_0\}$, such that for each $\alpha\in Q_1$ we have $\phi(t(\alpha))\circ M(\alpha)=N(\alpha)\circ \phi(s(\alpha))$. The category $\rep(Q,I)$ of finite-dimensional representations of $(Q,I)$ is Abelian, Artinian, Noetherian, has enough projectives and injectives, and contains only finitely many simple objects, seen as follows.

The (isomorphism classes of) simple objects in $(Q,I)$ are in bijection with the vertices of $Q$. For each $x\in Q_0$, the corresponding simple $S^x$ is the representation with $(S^x)_x=\C,\ (S^x)_y=0\mbox{ for all }y\in Q_0\setminus\{x\}$. 

For each $x\in Q_0$, we let $P^x$ (resp.\ $I^x$) denote the \defi{projective cover} (resp.\ \defi{injective envelope}) of $S^x$, as constructed in \cite[Section III.2]{ass-sim-sko}. In particular, for $y\in Q_0$, the dimension of $(P^x)_y$ (resp.\ $(I^x)_y$) is given by the number of paths from $x$ to $y$ (resp.\ from $y$ to $x$), considered up to the relations in $I$.

\subsection{Borel--Weil--Bott theorem}\label{sec:bott}

In this section, we present a special case of the Borel-Weil-Bott theorem that we use in Section \ref{subsec:ring}. For more details, see \cite{jerzy}.

Throughout we denote by $\Gr$ the Grassmannian $\Gr(3,6)$. We start with the tautological sequence of locally free sheaves on $\Gr$:
\begin{equation}\label{eq:taut}
0 \to \R \to W\oo \O_{\Gr} \to \Q \to 0.
\end{equation}
Let $\a = (\a_1 \geq \a_2 \geq \a_3)$ and $\b=(\b_1 \geq \b_2 \geq \b_3)$ be (dominant) weights in $\Z^3$. We describe the cohomology of the bundle
\[\SS_\a \Q \oo \SS_\b \R.\]
Put $\ll=(\a_1,\a_2,\a_3,\b_1,\b_2,\b_3)\in \Z^6$, and let $\rho=(5,4,3,2,1,0)$. Write $\sort(\ll+\rho)$ for the sequence of integers obtained by arranging the entries of $\ll+\rho$ in nonincreasing order. Let $\tilde{\ll}=\sort(\ll+\rho)-\rho$. By the Borel-Weil-Bott theorem \cite[Chapter 4]{jerzy}, we have the following.

\begin{theorem}\label{thm:bott}
If $\ll+\rho$ has repeated entries, then $H^i(\Gr, \SS_\a \Q \oo \SS_\b \R)=0$, for all $i\geq 0$. Otherwise, if $l$ denotes the length of the (unique) permutation that sorts $\ll+\rho$, then we have
\[
H^i(\Gr, \SS_\a \Q \oo \SS_\b \R)=
\begin{cases}
\SS_{\tilde{\ll}}W & i=l;\\
0 & i\neq 0.
\end{cases}
\]
\end{theorem}

\medskip

\section{The category of equivariant coherent $\D$-modules}\label{sec:cat}

As usual, $G=\GL(V)$ acts on $X=\bw^3 V$ with $V\cong \C^6$.

\subsection{Fundamental groups of orbits}\label{subsec:funds}

Given a $G$-orbit $O \cong G/H$ of $V$, we call the finite group $H/H^0$ the component group of $O$ (here $H^0$ denotes the connected component of $H$ containing the identity). If the orbit $O$ is simply-connected, then its component group is trivial.

We proceed by determining the fundamental groups of all orbits.

\begin{lemma}\label{lem:O}
The orbits $O_1,O_2,O_3$ are simply-connected.
\end{lemma}

\begin{proof}
It is easy to see that the $\GL_6$-orbits $O_1,O_2,O_3$ are also $\SL_6$-orbits. The generic $\GL_6$-stabilizers for these orbits have been computed in \cite[Section 10.4]{igusa}. From these descriptions we see that the $\SL_6$-stabilizers are connected. Since $\SL_6$ is connected and simply-connected, the claim now follows from the long exact sequence of homotopy groups obtained from the fibration $\SL_6\to O_i$ ($i=1,2,3$).
\end{proof}

\begin{remark}
The result above for $O_1$ follows also by \cite[Lemma 4.13]{catdmod} since $O_1$ is the orbit of the highest weight vector. The result for $O_2$ follows also using the desingularization $Z\to \ol{O_2}$ considered in \cite[Proposition 7.3.8]{jerzy}, and noticing that $Z \setminus O_2$ has codimension $\geq 2$.
\end{remark}

\begin{lemma}\label{lem:ratsing}
The variety $\ol{O_3}$ is normal with rational singularities.
\end{lemma}

\begin{proof}
Let $f\in S$ be the semi-invariant introduced in Section \ref{sec:intro}. By (\ref{eq:roots}) the polynomial $b_f(s)/(s+1)$ has no roots $\geq -1$. Thus, by \cite[Theorem 0.4]{saito} the variety $\ol{O_3}=V(f)$ has rational singularities.
\end{proof}

\begin{lemma}\label{lem:O4}
The fundamental group of $O_4$ is $\Z$ and its component group is $\Z/2\Z$.
\end{lemma}

\begin{proof}
The first assertion follows from \cite[Lemma 4.11]{catdmod} and Lemma \ref{lem:ratsing}. Note that the generic stabilizer is computed explicitly in \cite[Section 10.4]{igusa} as well as in \cite[Section 5]{saki}.
\end{proof}

\subsection{Classification of the simple equivariant $\D$-modules and the quiver structure of $\opmod_{G}(\D_X)$}

Recall the semi-invariant $f\in S$ introduced in Section \ref{sec:intro}. We determine the filtrations of $S_f$ and $S_f\cdot \sqrt{f}$ using the $b$-function of $f$. From this we obtain the explicit construction of each simple equivariant $\D$-module. We conclude the section by proving the Theorem on Quiver Structure.

We have the following filtrations of $\D$-modules in $S_f$ and $S_f\cdot \sqrt{f}$:

\begin{equation}\label{eq:filtrations}
0\sub S \sub \D f^{-1} \sub \D f^{-5}, \;\;\; 0\sub \D f^{-3/2}\sub \D f^{-5/2} \sub \D f^{-7/2}.
\end{equation}

We show now that the $6$ successive quotients above give all the (isomorphism classes of) simple equivariant $\D$-modules. The $6$ successive quotients are nonzero according to \cite[Proposition 4.9]{catdmod}, as from (\ref{eq:roots}) the roots of the $b$-function of $f$ are $-1$, $-5/2$, $-7/2$, and $-5$. Moreover, these $6$  $\D$-modules have unique simple quotients (again by \cite[Proposition 4.9]{catdmod}), which are nonisomorphic since they have different semi-invariant weights. But from Section \ref{subsec:funds} we deduce that there are in total six nonisomorphic simple equivariant $\D$-modules (see \cite[Remark 11.6.2]{hot-tak-tan}. Since each semi-invariant weight occurs in $S_f$ or $S_{f}\cdot \sqrt{f}$ with multiplicity one, this shows that the $6$ successive quotients are in fact all the simple $\D$-modules, and we have determined the composition series of $S_f$ and $S_{f}\cdot \sqrt{f}$, both of which have holonomic length $3$.

Given a $\D$-module $M$, we write $\charC(M)$ for its characteristic cycle (see \cite{kashi}), which is a formal linear combination of the irreducible components of its characteristic variety counted with multiplicities.

\begin{theorem}\label{thm:simples}
The (isomorphism classes) of simple equivariant $\D$-modules are given as follows:
\begin{itemize}
\item[(0)] The simple supported on $O_0$ is $E\cong \D f^{-5}/ \D f^{-1}$, and $\charC(E)=[\ol{T^\ast_{O_0} V}]$.
\item[(1)] The simple supported on $\ol{O_1}$ is $D_1\cong \D f^{-7/2}/\D f^{-5/2}$, and $\charC(D_1)=[\ol{T^\ast_{O_1}V}]+[\ol{T^\ast_{O_0} V}]$.
\item[(2)] The simple supported on $\ol{O_2}$ is $D_2\cong \D f^{-5/2}/\D f^{-3/2}$, and $\charC(D_2)=[\ol{T^\ast_{O_2}V}]$.
\item[(3)] The simple supported on $\ol{O_3}$ is $D_3\cong \D f^{-1}/S$, and $\charC(D_3)=[\ol{T^\ast_{O_3}V}]+[\ol{T^\ast_{O_2}V}]+[\ol{T^\ast_{O_1} V}]$.
\item[(4)] The simples supported on $\ol{O_4}$ are $S=\C[V]$ and $B_4\cong \D f^{-3/2}$ with $\charC(B_4)=[\ol{T^\ast_{O_4}V}]+[\ol{T^\ast_{O_3} V}]$.
\end{itemize}

\end{theorem}

\begin{proof}
The claim for $E$ follows from (\ref{eq:charE}). To show that $D_3\cong \D f^{-1}/S$, we note that $D_3$ is the unique simple $\D$-submodule of the local cohomology $H^1_f(S)\cong S_f/S$. Since $\D f^{-1}/S$ is such a simple, it must be isomorphic to $D_3$.

By (\ref{eq:fourier}), the Fourier transform $\F(D_3)$ has a weight $(-8)^6$. Since $D_3$ is the only simple having such a weight (corresponding to $f^{-4}$), we must have $\F(D_3)\cong D_3$.

Clearly, the simple $\D f^{-3/2}$ has full support $V$, hence we must have $B_4 \cong \D f^{-3/2}$.

We recall some facts on the interplay between projective duality and Fourier transform (for more details, see \cite[Section 4.3]{catdmod}). Since $G$ is reductive, acting on $X$ with finitely many orbits, projective duality induces an involution on the set of orbits (by identifying $X$ with $X^*$). Moreover, the irreducible components of the characteristic variety of an equivariant $\D$-module are closures of conormal bundles to orbits, and the Fourier transform interchanges these components by way of projective duality. 

It is known that the hypersurface $\ol{O_3}$ is the projective dual to $\ol{O_1}$ (see \cite{knopmenzel}), and so $\ol{O_2}$ is self-dual. Assume by contradiction that $\ol{T^\ast_{O_0} V}$ is not an irreducible component of the characteristic variety of $D_1$. Then $\F(D_1)$ is a simple equivariant $\D$-module supported on $\ol{O_3}$ (see \cite[Section 4.3]{catdmod}), hence it must be $D_3$, contradicting $\F(D_3)\cong D_3$. Hence, $\ol{T^\ast_{O_0} V}$ is a component of the characteristic variety of $D_1$. Then $\F(D_1)$ has full support $X$ (again by \cite[Section 4.3]{catdmod}), so $\F(D_1)\cong B_4$, and also $\F(D_2)\cong D_2$.

Since $D_1\cong\F(B_4)$ has a weight $(-7^6)$, we have  $D_1\cong \D f^{-7/2}/\D f^{-5/2}$, and then $D_2\cong \D f^{-5/2}/\D f^{-3/2}$.

All simples come from local systems of rank one. Hence, to finish the description of the characteristic cycles, we are left to show that $[\ol{T^\ast_{O_2}V}]$ appears in $\charC(D_3)$ with multiplicity one. For the remainder of the proof, put $Z=\ol{T^\ast_{O_2}V}$.

We know that $\ol{T^\ast_{O_3}V}$ and $\ol{T^\ast_{O_1} V}$ are components of the characteristic variety of $D_3$ (since $\F(D_3) \cong D_3$), and their intersection has codimension $\geq$ 2, since there is no edge between them in the holonomy diagram  \cite[Figure 8-1]{kimu}. It follows by \cite[Theorem 6.7]{macvil} that $Z$ must be a component of characteristic variety of $D_3$.

Now we show that the multiplicity of $Z$ in $\charC(D_3)$ is at most one. Consider the projective cover $P=P(\lambda)$ of $D_3$ in the category $\opmod_G(\D_X)$, with $\lambda=-2^6$, where we use the notation as in \cite[Lemma 2.1]{catdmod}. Let $\mu: T^* X \to \mathfrak{g}^*$ be the moment map and denote by $\operatorname{mult}_Z \mathcal{O}_{\mu^{-1}(0)}$ the multiplicity of $Z$ along the scheme-theoretic zero-fiber (see \cite[Section 3.1]{catdmod} for more details). It is enough to show that the multiplicity of $Z$ in $\charC(P)$ is at most one. By the proof of \cite[Proposition 3.14]{catdmod}, the latter multiplicity is bounded above by $\operatorname{mult}_Z \mathcal{O}_{\mu^{-1}(0)}$. By \cite[Section 8]{kimu}, the variety $Z$ has a dense $G$-orbit. Hence, by \cite[Lemma 3.12]{catdmod} we get $\operatorname{mult}_Z \mathcal{O}_{\mu^{-1}(0)}=1$, finishing the proof.
\end{proof}

\begin{remark} \label{rem:localb}
One can also obtain the supports of the simples by computing the \textit{local} $b$-functions of $f$ at each orbit. These can be computed for example as outlined in \cite[Example 2.12]{bub} and it also follows from the description of the holonomy diagram in \cite[Section 8]{kimu}. For example, if $v_{(2,1,1,1,1,0)}$ denotes (up to a nonzero constant) the polynomial of highest weight $(2,1,1,1,1,0)$, and $\partial_{(0,-1,-1,-1,-1,-2)}$ denotes the constant differential operator of highest weight $(0,-1,-1,-1,-1,-2)$, then we have (see \cite[Theorem 2.5]{bub})
\[\partial_{(0,-1,-1,-1,-1,-2)} \cdot f^{s+1} = (s+1)(s+5/2) \cdot v_{(2,1,1,1,1,0)} f^s.\]
\end{remark}

\vspace{0.1in}

We note that for each $G$-orbit $O$, the irreducible representations of the component group of $O$ are self-dual. This implies that the holonomic duality functor fixes all the simple equivariant $\D$-modules. Thus, the Classification Theorem stated in Section \ref{sec:intro} follows. We proceed with proving the next claim from Section \ref{sec:intro}:

\begin{proof}
[Proof of Theorem on Quiver Structure]

By \cite[Lemma 2.4]{bindmod} we have that $S_f$ is the injective hull of $S$ and $S_f\cdot \sqrt{f}$ is the injective hull of $B_4$ in the category  $\opmod_G(\D_X)$. From the filtrations (\ref{eq:filtrations}) of $S_f$ and $S_f\cdot \sqrt{f}$, we obtain extensions between the simples in $\opmod_G(\D_X)$. Moreover, these extensions are nonsplit, as it follows by \cite[Lemma 2.8]{catdmod} (applied to the respective semi-invariant weights). This yields all the nontrivial paths in $(Q,I)$ going into $s$ and $b_4$, obtaining the arrows $\b_0$, $\b_1$, $\d_0$, $\d_1$, and that the paths $\b_1 \b_0$ and $\d_1\d_0$ are nonzero (i.e., not in $I$). Applying the holonomic duality functor $\bb{D}$, we obtain the arrows $\a_0$, $\a_1$, $\g_0$, and $\g_1$. Applying the Fourier transform, we see that there are no other arrows adjacent to the vertices $s,e,b_4,d_1$. Moreover, by \cite[Corollary 3.9]{catdmod} together with Fourier transform, we see that all 2-cycles must be zero.

We are left to show that there are no arrows between $d_2$ and $d_3$. Since $H^1_f(S)=S_f/S$ is the injective hull of $D_3$ in the category  $\opmod_G^{\ol{O_3}}(\D_X)$ (see \cite[Lemma 3.11]{catdmod}) of equivariant $\D$-modules supported on $\ol{O_3}$, and it does not have $D_2$ as a composition factor, we see that there are no arrows from $d_2$ to $d_3$. By holonomic duality $\bb{D}$, there are no arrows from $d_3$ to $d_2$ either, finishing the proof.
\end{proof}

\section{Characters of equivariant $\D$-modules}
\label{sec:character}

Since $G$ acts on $X$ with finitely many orbits, for an equivariant coherent $\D$-module $M$ the $G$-representation $[M]$ is admissible by \cite[Proposition 3.14]{catdmod}. In this section, we describe explicitly the $G$-module structure of all the simple equivariant $\D$-modules. 

The characters of $S$ and $E$ are given by the rational expressions (\ref{eq:charS}) and (\ref{eq:charE}), respectively. The character of $S_f$ is given by
\begin{equation}\label{eq:charSf}
[S_f] = \lim_{n\to \infty} [f^{-n}\cdot S] = \frac{1}{(1-e^{(1,1,1,0,0,0)})(1-e^{(2,1,1,1,1,0)})(1-e^{(2,2,2,1,1,1)})(1-e^{(3,3,2,2,1,1)})} \cdot e^{(2,2,2,2,2,2) \Z},
\end{equation}
where $e^{(2,2,2,2,2,2) \Z}=\sum_{i\in \Z} e^{(2i^6)}$. Hence, we obtain the character of $D_3$ from $[D_3]=[S_f]-[S]-[E]$.

Similarly, from $\F(D_1) \cong B_4$ we have by (\ref{eq:fourier}) the relation $[D_1]=[B_4]^\ast\cdot e^{(-10^6)}$. Clearly, $[S_f \cdot \sqrt{f}]= [S_f] \cdot e^{(1^6)}$ and $[D_2]=[S_f  \cdot \sqrt{f}]-[B_4]-[D_1]$. Hence, to complete the description of the characters of the simple equivariant $\D$-modules, it suffices to find the character of $B_4$.

\begin{theorem}\label{thm:charB4}
The $\GL_6$-character of $B_4$ is given by
\[[B_4] =\frac{e^{(-3^6)}}{(1-e^{(0,0,0,-1,-1,-1)})(1-e^{(1,1,0,0,-1,-1)})(1-e^{(1,1,1,0,0,0)})(1-e^{(2,1,1,1,1,0)})(1-e^{(2^6)})}.\]
\end{theorem}

In order to prove the theorem above, we begin with some preliminary results. In the following lemma, $G$ can be any reductive group acting on a vector space $X$, and $h\in S=\C[X]$ is a semi-invariant of weight $\sigma$.

\begin{lemma}\label{lem:charDf}
Let $q \in \bb{Q}$ such that $\sigma^q$ is an algebraic character, and consider the $G$-equivariant $\D$-module $M=\D_X\cdot h^{q}$. 
For $k\geq 0$, denote by $I_k$ (resp. $\tilde{I}_k$) the following $G$-equivariant ideal of $S$ (resp. of $S/(h)$):
\[ I_k =(M: h^{q-k})= \left\{ s \in S \mid s \cdot h^{q-k} \in M\right\}\,\,\, \mbox{ and } \,\,\, \tilde{I}_k = (I_k+(h))/(h) .\]
Then the $G$-character of $M$ is given by
\[[M]\cdot \sigma^{-q} =  \lim_{k\to \infty} [I_k]\cdot\sigma^{-k} = [S]+ \sum_{k\geq 1} [\tilde{I}_k]\cdot \sigma^{-k}.\]
\end{lemma}

\begin{proof}
Clearly, $I_0=S$. Moreover, we have $I_{k} \subset I_{k-1}$ and $I_k \cap (h) = h \cdot I_{k-1}$, for all $k\geq 1$. We obtain an exhaustive filtration of $M$
\[S\cdot h^{q} \subset I_1\cdot h^{q-1} \subset \dots \subset I_{k-1}\cdot h^{q-(k-1)} \subset I_k\cdot h^{q-k} \subset \dots\]
We have $\tilde{I}_k \cong I_k/(I_{k}\cap (h)) = I_k/h I_{k-1}$. Hence, we obtain
\[ [M]\cdot \sigma^{-q} = \lim_{k\to \infty} [I_k \cdot h^{-k}] = [S]+ \sum_{k\geq 1}  [I_k/h I_{k-1}]\cdot \sigma^{-k}=[S]+ \sum_{k\geq 1} [\tilde{I}_k]\cdot \sigma^{-k}.\]
\end{proof}

We return to our situation $G=\GL_6$ and $X=\bw^3 V$. We keep the notation from Lemma \ref{lem:charDf}.

\begin{lemma}\label{lem:charIk}
Take $M=B_4=\D_X\cdot f^{-3/2}$ and the $S/(f)$-ideals $\tilde{I}_k$ as in Lemma \ref{lem:charDf}. For $k\geq 1$, we have
\[[\tilde{I}_k] = \frac{1}{(1-e^{(1,1,1,0,0,0)})(1-e^{(2,1,1,1,1,0)})} \cdot \sum_{\substack{m,n \geq 0 \\ m+n\geq k}} e^{m(2,2,2,1,1,1)}\cdot e^{n(3,3,2,2,1,1)}.\]
\end{lemma}

\begin{proof}
We start with describing $I_k$ (and $\tilde{I}_k$) in Lemma \ref{lem:charDf} in more detail. Let $S^*=\C[\partial_{i,j,k} \mid 1\leq i < j < k \leq 6]$ denote the polynomial ring in the dual variables. Let $V_{(0,-1,-1,-1,-1,-2)}$ denote the (unique) representation of $S^*$ of highest weight $(0,-1,-1,-1,-1,-2)$.  By Remark \ref{rem:localb}, the ideal $(V_{(0,-1,-1,-1,-1,-2)})$ generated by the elements in $V_{(0,-1,-1,-1,-1,-2)}$ annihilates $f^{-3/2}$. Moreover, this ideal defines the closure of the orbit of the highest weight vector in $X^*$ of weight $(0,0,0,-1,-1,-1)$. We have a $G$-decomposition of the coordinate ring (see \cite[Chapter 5, Exercise 8]{jerzy})
\[[S^*/(V_{(0,-1,-1,-1,-1,-2)})] = \frac{1}{1-e^{(0,0,0,-1,-1,-1)}}.\]
Let $\partial_1$ denote the constant coefficient differential operator of degree $1$ of highest weight $(0,0,0,-1,-1,-1)$. The above discussion shows that the ideal $I_k$ is generated by the elements in the irreducible representation with highest weight vector $(\partial_1^k \cdot f^{-3/2})\cdot f^{3/2+k}$ of weight $(2k,2k,2k,k,k,k)$. Clearly, this vector is not in $(f)$.

We note that by (\ref{eq:charS}) the character of $S/(f)$ is multiplicity-free, i.e., any representation of $G$ occurs in $[S/(f)]$ with multiplicity at most one. Indeed, this follows from the fact that the tuples $(1,1,1,0,0,0)$, $(2,1,1,1,1,0)$, $(2,2,2,1,1,1)$, and $(3,3,2,2,1,1)$ are linearly independent over $\mathbb{Z}$. 

Hence, $\tilde{I}_k$ is the ideal of $S/(f)$ generated by the elements in the (unique) representation of highest weight $(2k,2k,2k,k,k,k)$.

The multiplication map in $S$ gives a surjection $\SS_{(1,1,1,0,0,0)} \oo \SS_{(2,2,2,1,1,1)} \to \SS_{(3,3,2,2,1,1)}$ (one can see this for example noticing that $e^{(3,3,2,2,1,1)}$ does not appear in $\C[\ol{O_2}]$ by using the desingularization in \cite[Proposition 7.3.8]{jerzy}). We deduce that $\tilde{I}_k$ contains all the highest weight vectors of weight 
\[e^{a(1,1,1,0,0,0)} \cdot e^{b(2,1,1,1,1,0)} \cdot e^{m(2,2,2,1,1,1)} \cdot e^{n(3,3,2,2,1,1)}, \,\,\, \mbox{ where } a,b \geq 0,\, m+n \geq k.\] 
On the other hand, if $m+n < k$ in the weight above, then the last entry of the weight is smaller than $k$, hence the weight cannot appear in $\tilde{I}_k$. This yields the result. 
\end{proof}

\begin{proof}[Proof of Theorem \ref{thm:charB4}]
Using (\ref{eq:charS}) and Lemma \ref{lem:charIk}, we have according to Lemma \ref{lem:charDf}
\[[B_4]\cdot e^{(3^6)}  = [S]+ \sum_{k\geq 1} [\tilde{I}_k]\cdot e^{(-2k^6)} =  \frac{1}{(1-e^{(1,1,1,0,0,0)})(1-e^{(2,1,1,1,1,0)})} \, \times \] 
\[\times \Bigl(\sum_{a,b,c \geq 0} e^{a(0,0,0,-1,-1,-1)}\cdot e^{b(1,1,0,0,-1,-1)}\cdot e^{(a+b+c)(2^6)}+\!\!\sum_{\substack{m,n,k \geq 0 \\ m+n\geq k}} e^{m(0,0,0,-1,-1,-1)}\cdot e^{n(1,1,0,0,-1,-1)}\cdot e^{(m+n-k)(2^6)}\Bigr)= \smallskip\]
\[=\frac{1}{(1-e^{(0,0,0,-1,-1,-1)})(1-e^{(1,1,0,0,-1,-1)})(1-e^{(1,1,1,0,0,0)})(1-e^{(2,1,1,1,1,0)})(1-e^{(2^6)})}.\]
\end{proof}

\section{Local cohomology}
\label{sec:loccoh}

In this section, we determine the all the (iterated) local cohomology modules of the simple equivariant $\D$-modules supported in the orbit closures.

\subsection{Local cohomology of $S$}\label{subsec:ring}

The goal in this section is to prove the Theorem on Local Cohomology as stated in the Introduction \ref{sec:intro}. Part (3) follows by Theorem \ref{thm:simples} since $H^1_{\ol{O_3}}(S)=S_f/S$. The nontrivial parts are (1) and (2). 

Since $H^{5}_{\ol{O_2}}(S)$ (resp. $H^{10}_{\ol{O_1}}(S)$) is the injective envelope of $D_2$ in $\opmod_G^{\ol{O_2}}(\D_X)$ (resp. of $D_1$ in $\opmod_G^{\ol{O_1}}(\D_X)$) by \cite[Lemma 3.11]{catdmod}, the claim about their structures follows by description of the quiver of $\opmod_G(\D_X)$.  In fact, we have an isomorphism 
\begin{equation}\label{eq:HB4}
H^{5}_{\ol{O_2}}(S)\cong S_f \sqrt{f}/\D f^{-3/2} \cong H^1_{\ol{O_3}}(B_4),
\end{equation}
where the last isomorphism follows immediately from the \v{C}ech cohomology description of local cohomology. 

We now proceed with part (1). Since  $\ol{O_1}$ is the cone over a smooth projective variety, there are several results in this direction relating local cohomology to singular cohomology \cite{Ogus}, \cite{garsab}, \cite{switala2015lyubeznik},  \cite{LSW}.

\begin{prop}\label{prop:locO1}
The following are all the nonzero local cohomology modules of $S$ with support in $\ol{O_1}$:
\[H^{10}_{\ol{O_1}}(S) = D_1, \;\; H^{13}_{\ol{O_1}}(S) = E, \;\; H^{15}_{\ol{O_1}}(S) = E.\]
\end{prop}

\begin{proof}
By the discussion above, we just need to determine how many copies of $E$ appear in $H^j_{\ol{O_1}}(S)$ for $j>10$. Write $\beta_i$ for the Betti numbers of the algebraic de Rham cohomology of the Grassmannian. By \cite[Main Theorem 1.2]{switala2015lyubeznik} (see also \cite[Theorem]{garsab} and \cite[Theorem 3.1]{LSW}), we have $H^j_{\ol{O_1}}(S)=E^{\oplus (\beta_{20-j-1}-\beta_{20-j-3})}$ for $11\leq j\leq 17$, $H^{18}_{\ol{O_1}}(S)=E^{\oplus \beta_1}$, $H^{19}_{\ol{O_1}}(S)=E^{\oplus \beta_0-1}$, and $H^{20}_{\ol{O_1}}(S)=0$. These Betti numbers are determined by the following Gaussian binomial coefficient \cite[Corollary 3.2.5]{manivel}:
$$
\sum_{i=0}^{18}\beta_i\cdot q^i=\binom{6}{3}_{q^2}=1+q^2+2q^4+3q^6+3q^8+3q^{10}+3q^{12}+2q^{14}+q^{16}+q^{18}.
$$
We conclude immediately that $H^{18}_{\ol{O_1}}(S)=H^{19}_{\ol{O_1}}(S)=0$. Computing the differences $\beta_{20-j-1}-\beta_{20-j-3}$ for $11\leq j\leq 17$ yields the desired result.
\end{proof}

We are left with proving part (2) of the Theorem on Local Cohomology, which we devote the rest of the section to.\medskip

We outline the strategy that we use to compute all the local cohomology modules of $S$ supported on $\ol{O_2}$. First, it is enough to find the multiplicities of the relevant semi-invariant ``witness'' weights in each local cohomology module. For this, we first consider a desingularization of the hypersurface $\ol{O_3}$ which we then use to locate the witness weights in the local cohomology modules of $S/(f)$ supported on $\ol{O_2}$. From the knowledge obtained for $H^i_{\ol{O_2}}(S/(f))$ we then proceed to recover $H^i_{\ol{O_2}}(S)$ using the associated long exact sequence. \medskip

Throughout we denote by $\Gr$ the Grassmannian $\Gr(3,6)$ and use the notation in Section \ref{sec:bott}.

As mentioned in the proof of Theorem \ref{thm:simples}, $\ol{O_3}$ is the projective dual of the highest weight orbit, hence a discriminant in the sense of \cite{jerzydisc}. Thus, it has a desingularization as the total space $Z=\Tot(\eta^*)$ of a bundle $\eta$ of 1-jets on $\Gr$, as described in \cite[Section 1]{jerzydisc}. The space $Z$ is a subbundle of the trivial bundle $\Gr \times \bw^3 V$, and we denote the first and the second projection (which yields the desingularization of $\ol{O_3}$) by 
\[p:Z \lra \Gr, \;\;\;\;\; q : Z \lra \ol{O_3}.\]
We denote by $\xi$ the locally free sheaf on $\Gr(3,6)$ corresponding to the quotient bundle obtained from the inclusion $Z\subset \Gr \times \bw^3 V$. Hence, we have the following exact sequence of locally free sheaves on $\Gr$:
\[ 0\to \xi \to \bw^3 W\oo \O_{\Gr} \to \eta \to 0.\] 
Moreover, the sheaf $\xi$ fits into the following exact sequence (see \cite[Section 1]{jerzydisc}, or \cite[Section 9.3]{jerzy}):
\[
0 \to \bw^3 \R \to \xi \to \Q\oo \bw^2 \R \to 0.
\]
It follows that $\eta$ fits into the exact sequence
\begin{equation}\label{eq:eta}
0 \to \bw^2 \Q\oo \R \to \eta \to \bw^3 \Q \to 0.
\end{equation}

The difficulty in the following calculations stems from the fact that the sequence (\ref{eq:eta}) does not split. Nevertheless, for each $k\geq 1$ we have the exact sequence:
\begin{equation}\label{eq:sym}
0 \to \Sym_k (\bw^2 \Q \oo \R) \to\Sym_k \eta \to \Sym_{k-1} \eta \otimes \bw^3 \Q \to 0.
\end{equation}

The left-hand side can be decomposed according to Cauchy's formula (\cite[Corollary 2.3.3]{jerzy}) (keeping in mind $\bw^2 \Q \cong \Q^\ast \oo \bw^3 \Q$):
\begin{equation}\label{eq:cauchy}
\Sym_k (\bw^2 \Q \oo \R) = \bo_{|\ll|=k} \SS_\ll \Q^\ast \oo \SS_\ll \R \oo (\bw^3 \Q)^{\oo k} =\bo_{|\ll|=k} \SS_{(\ll_1+\ll_2, \ll_1+\ll_3, \ll_2+\ll_3)} \Q \oo \SS_{(\ll_1,\ll_2,\ll_3)} \R, 
\end{equation}
where the sum is over all partitions $\ll = (\ll_1, \ll_2 ,\ll_3)$ of $k$.

\begin{lemma}\label{lem:nonsplit}
We have $H^0(\Gr,\Sym \eta)=S/(f)$, and $H^i(\Gr, \Sym \eta) = 0$ for $i\geq 0$. Moreover, $\eta$ is characterized as the unique (up to scalar) nonsplit extension in the sequence (\ref{eq:eta}).
\end{lemma}

\begin{proof}
For all $i\geq 0$ we have (since the variety $\ol{O_3}$ and the map $p$ are affine)
\[\bb{R}^i q_\ast \O_Z \cong H^i(Z,\O_Z) \cong H^i(\Gr, p_* \O_Z) = H^i(\Gr, \Sym \eta).\]
By Lemma \ref{lem:ratsing} $\ol{O_3}$ has rational singularities, hence the first claim.
Using Theorem \ref{thm:bott} we see that whenever $\ll_1 \geq 2+\ll_2+\ll_3$, the cohomology of the summand in (\ref{eq:cauchy}) corresponding to the partition $(\ll_1,\ll_2,\ll_3)$ contributes to $H^1(\Gr,\Sym_k (\bw^2 \Q \oo \R) )$. Since the connecting homomorphism $H^0(\Gr,  \Sym_{k-1} \eta \oo \bw^3 \Q)\to H^1(\Sym_k (\bw^2 \Q \oo \R))$ from (\ref{eq:sym}) is surjective, this shows that it is nonzero when $k\geq 2$. In particular, the sequences (\ref{eq:sym}) do not split for $k\geq 2$. Hence (\ref{eq:eta}) does not split, and this characterizes $\eta$ since by Theorem \ref{thm:bott}
\[\Ext(\bw^3 \Q, \bw^2 \Q\oo \R) \cong \Ext(\O_{Gr}, \Q^* \oo \R) \cong H^1(\Gr, \Q^* \oo \R) = \C.\]
\end{proof}

Let $U=q^{-1}(O_3)$, which is an open subset in $Z$. Since $q$ is a $G$-equivariant birational isomorphism, we have $U\cong O_3$ as $G$-varieties.

\begin{prop}\label{prop:divisor}
We have $Z\setminus U = D$, where $D$ is a $G$-stable divisor on $Z$. Moreover, the ideal sheaf of $D$ is $p^* (\L)$, where $\L$ is the line bundle on $\Gr$:
\[ \L = (\bw^3 \Q)^{\oo 2} \oo \bw^3 \R.\]
\end{prop}

\begin{proof}
The sequence (\ref{eq:eta}) gives an inclusion $\Sym_{3} (\bw^2 \Q \oo \R) \hookrightarrow \Sym_{3} \eta$. In the decomposition (\ref{eq:cauchy}) (for $k=3$) the summand corresponding to the partition $\ll=(1,1,1)$ is $\L$, thus proving that we have an equivariant inclusion $\L \hookrightarrow \Sym \eta$. By adjunction of $p_*$ and $p^*$, this in turn yields an inclusion
\[ p^*(\L) \hookrightarrow \O_Z,\]
yielding the $G$-stable divisor $D$. To finish the proof, we need to show that all closed $G$-stable proper subsets in $Z$ are contained in $D$. 

Let $P$ be the parabolic subgroup of $\GL_6(\C)$ such that $\GL_6(\C)/P \cong \Gr(3,6)$, which has the Levi decomposition $P\cong L \ltimes U$ with $L\cong \GL_3(\C) \times \GL_3(\C)$ and $U\cong \Hom(\C^3, \C^3)$. From (\ref{eq:eta}) we can write $Z\cong G\times_P N$, where $N$ fits into the exact sequence of $P$-modules
\begin{equation}\label{eq:pee}
0 \to \bw^3 \C^3 \to N \to \bw^3 \C^3 \oo \Hom(\C^3,\C^3) \to 0. 
\end{equation}
There is a one-to-one correspondence between the $G$-stable closed subsets in $Z$ and the $P$-stable closed subsets in $N$. Clearly, $D$ corresponds to the $P$-stable divisor $D'$ on $N$ obtained by taking determinant $N \to \bw^3 \C^3 \oo \Hom(\C^3,\C^3) \xrightarrow{\det} \C$. We are left to show that all $P$-stable closed proper subsets in $N$ are contained in $D'$.

We first identify the $L$-stable closed subsets in $N$. From (\ref{eq:pee}) we see that as an $L$-module, 
\[N\cong  \bw^3 \C^3 \,\bo\, \bw^3 \C^3 \oo \Hom(\C^3,\C^3)\] By inspection, $L$ acts on $N$ with $8$ orbits, and the orbits are of the form $O'_{1,j}=(\C^\ast,M_j)$, $O'_{0,j}=(0,M_j)$, for $j=0,1,2,3$, where $M_j$ denotes the set of matrices in $\Hom(\C^3,\C^3)$ of rank $j$. Since $L$ acts on $N$ with finitely many orbits, each irreducible $P$-stable closed subset of $N$ must be of the form $\ol{O'_{i,j}}$, for some $i\in \{0,1\}$ and $j\in \{0,1,2,3\}$. Note that each $\ol{O'_{i,j}}$ is contained in $\D'=\ol{O'_{1,2}}$ except when $j=3$. But $\ol{O'_{1,3}} = N$, and $\ol{O'_{0,3}}$ cannot be $P$-stable, otherwise the sequence (\ref{eq:pee}) would split, contradicting Lemma \ref{lem:nonsplit}.

\end{proof}

The following allows us to compute local cohomology in terms of cohomology on the Grassmannian. We keep the notation as in Proposition \ref{prop:divisor}.

\begin{prop}\label{prop:localopen}
For each $i\geq 2$, we have an isomorphism of $G$-modules
\[H_{\ol{O_2}}^i(S/(f)) \cong \varinjlim_k H^{i-1} (\Gr, \L^{-k}\oo\Sym \eta).\]
\end{prop}

\begin{proof}
Denote by $j:U \to Z$ the open embedding, which is affine according to Proposition \ref{prop:divisor}. We have
\[H_{\ol{O_2}}^i(S/(f)) \cong H^{i-1}(O_3,\O_{O_3})\cong  H^{i-1}(U,\O_{U})\cong H^{i-1}(Z, j_* \O_{U}).\]
In view of Proposition \ref{prop:divisor} we have
\[j_* \O_{U}=\varinjlim_k\; (p^*\L)^{-k}.\] 
Note that cohomology commutes with direct limits \cite[Proposition III.2.9]{hartshorne}. The result follows (since $p$ is affine) by
\[H^{i-1}(Z,  p^*\L^{-k}) \cong H^{i-1}(\Gr, p_* p^*\L^{-k}) \cong H^{i-1}(\Gr, \L^{-k} \oo \Sym \eta).\]
\end{proof}

We have the following exact sequence of $G$-equivariant $S$-modules:
\begin{equation}\label{eq:exact}
0\to S \oo (2^6) \xrightarrow{\;\;f\;\;} S \lra S/(f) \to 0.
\end{equation}
The associated long exact sequence in local cohomology gives

\begin{equation}\label{eq:long}
\dots \to H^{i-1}_{\ol{O_2}}(S/(f)) \to H^i_{\ol{O_2}}(S)\oo (2^6) \xrightarrow{\;f\;} H^i_{\ol{O_2}}(S) \to H^i_{\ol{O_2}}(S/(f)) \to H^{i+1}_{\ol{O_2}}(S)\oo (2^6) \to \dots
\end{equation}

We summarize some of the preliminary knowledge about the modules  $H^i_{\ol{O_2}}(S)$. Since $\ol{O_2}$ has codimension $5$ in $X$, we have $H^i_{\ol{O_2}}(S)=0$ for $i\leq 4$, and so also $H^j_{\ol{O_2}}(S/(f))=0$ for $j\leq 3$.

As mentioned at the beginning of the section, we have  $H^{5}_{\ol{O_2}}(S) \cong S_f\sqrt{f}/\D f^{-3/2}$, which has the semi-invariant weight $(-5^6)$ that is annihilated under multiplication by $f$. In particular, we see from (\ref{eq:long}) that the only semi-invariant weight in $H^{4}_{\ol{O_2}}(S/(f))$ is $(-3^6)$ (with multiplicity 1).


Since in the quiver $(Q,I)$ there are no arrows between $d_1$ and $e$, for all $i\geq 6$ the $\D$-module $H^{i}_{\ol{O_2}}(S)$ must be a direct sum  
\begin{equation}\label{eq:de}
H^{i}_{\ol{O_2}}(S)\cong D_1 ^{\oplus a_i} \oplus E^{\oplus b_i}, \;\; \mbox{ for some } a_i, b_i \geq 0.
\end{equation}

Since $D_1 \cong \D f^{-7/2}/\D f^{-5/2}$ (resp. $E\cong \D f^{-5}/ \D f^{-4}$), the only semi-invariant weight of $D_1$ (resp. $E$) that is annihilated under multiplication by $f$ is $(-7^6)$ (resp. $(-10^6)$). Hence, using the long exact sequence (\ref{eq:long}) we see that in order to determine $a_i,b_i$ in (\ref{eq:de}), it is enough to track the semi-invariant weights $\chi_1 = (-5^6)$ (resp. $\chi_2=(-8^6)$) in $H^{i-1}_{\ol{O_2}}(S/(f))$. Summarizing, we obtained (using the notation in (\ref{eq:de})):

\begin{prop}\label{prop:prelim}
For $i\geq 5$, we have 
\[\la H^{i}_{\ol{O_2}}(S/(f)), (-5^6) \ra = a_{i+1}, \;\;\; \la  H^{i}_{\ol{O_2}}(S/(f)), (-8^6) \ra = b_{i+1}.\]
Also, for $i\leq 4$ we have 
\[\la H^{i}_{\ol{O_2}}(S/(f)), (-5^6) \ra = 0, \;\;\; \la  H^{i}_{\ol{O_2}}(S/(f)), (-8^6) \ra =0.\]
\end{prop}
 
\bigskip

\textbf{Notation.} From now on, $\chi$ will denote either of the characters $\chi_1 = (-5^6)$ or $\chi_2 = (-8^6)$. For a representation $M$ of $G$, we denote by $M^\chi$ its isotypical component corresponding to $\chi$.

\bigskip

Hence, by Propositions \ref{prop:localopen} and \ref{prop:prelim}, we are left to compute $\varinjlim_k H^{i} (\Gr, \L^{-k}\oo\Sym \eta)^\chi$. We first estimate cohomology by working with the associated graded of $\eta$:
\begin{equation}
\gr \eta= \bw^2 \Q \oo \R \, \bo \, \bw^3 \Q. 
\end{equation}

\begin{prop}\label{prop:gr}
For $d,i,k\geq 0$, the space 
\[H^{i} (\Gr, \L^{-k}\oo\Sym_d (\gr \eta))^\chi\] 
is nonzero only in the following situations (when it is 1-dimensional):
\begin{itemize}
\item[(1)] For $\chi=\chi_1$ with $d=3k-10$: $i=0,1,3,4$ for $k\geq 5$; and $i=5$ for $k\geq 4$.
\item[(2)] For $\chi= \chi_2$ with $d=3k-16$: $i=0,1,3,4$ for $k\geq 8$; $i=5,6$ for $k\geq 7$; and $i=8$ for $k\geq 6$.
\end{itemize}
\end{prop}

\begin{proof}
Let us denote $\chi=(x^6)$ (so $x$ is either $-5$ or $-8$). We have by (\ref{eq:cauchy}) the decomposition
\begin{equation}\label{eq:part}
 \L^{-k}\oo\Sym_d (\gr \eta)= \bo_{|\ll|\leq d} \SS_{d-2k-\ll_3, d-2k-\ll_2, d-2k-\ll_1}\Q \,\oo\, \SS_{\ll_1-k,\ll_2-k,\ll_3-k}\R.
 \end{equation}
Now we analyze using Theorem \ref{thm:bott} which partitions $\ll$ above yield $e^\chi$ in cohomology. Clearly, for the degrees to match we must have $3d-9k = 6x$, which gives $d=3k+2x$.

For a partition $\ll$ from (\ref{eq:part}), put $a=\ll_1 - k - x, b=\ll_2-k - x, c=\ll_3-k - x$. Then the numbers satisfy $a\geq b\geq c\geq -x-k,\; a+b+c\leq -x$, and the corresponding sequence from \ref{eq:part} is $(-c+x,-b+x,-a+x, a+x, b+x, c+x)$. Clearly, this sequence yields $(x^6)$ by the exchange rule if and only if the following sequence yields $(0^6)$:
\[ (-c,\,-b,\,-a,\,a,\,b,\,c).\]
By inspection, the following triplets $(a,b,c)$ with $a\geq b\geq c$ give $(0^6)$ above via $a+b+c$ exchanges:
\[(0,0,0);\;\; (1,0,0);\;\;(2,1,0);\;\;(2,2,0);\;\;(3,1,1);\;\;(3,2,1);\;\;(3,3,2);\;\;(3,3,3).\]
Analyzing the conditions $ k\geq -x-c,\; a+b+c\leq -x$ for $x=-5$ and $x=-8$ finishes the proof.
\end{proof}

The sequences that give cohomology above come from the following bundles in the decomposition (\ref{eq:part}):
\begin{equation}\label{eq:bundles}
\begin{aligned}
\SS_{x,x,x}\Q \oo \SS_{x,x,x} \R; \;\; \SS_{x,x,x-1}\Q \oo \SS_{x+1,x,x} \R; \;\; \SS_{x,x-1,x-2}\Q \oo \SS_{x+2,x+1,x} \R; \;\;\SS_{x,x-2,x-2}\Q \oo \SS_{x+2,x+2,x} \R ; \\
\SS_{x-1,x-1,x-3}\Q \oo \SS_{x+3,x+1,x+1} \R; \;\;\SS_{x-1,x-2,x-3}\Q \oo \SS_{x+3,x+2,x+1} \R; \;\;\SS_{x-2,x-2,x-3}\Q \oo \SS_{x+3,x+3,x+2} \R; \;\;,
\end{aligned}
\end{equation}
where $x=-5$ (bundles in (\ref{eq:bundles}) from first row) or $x=-8$ (all bundles from (\ref{eq:bundles})).

The space $H^{i} (\Gr, \L^{-k}\oo\Sym \eta)^\chi$ is generally smaller than $H^{i} (\Gr, \L^{-k}\oo\Sym_d (\gr \eta))^\chi$, as cancellations might occur in the latter between two terms for consecutive $i$. In order to describe the connecting homomorphisms leading to these potential cancellations, we use the sequence (\ref{eq:sym}) recursively. We note that by Propositions \ref{prop:localopen} and \ref{prop:prelim} we are only interested in cohomology for $i\geq 4$, so we ignore the cases $i=0,1$ in Proposition \ref{prop:gr}. We still keep track of $i=3$ though, since it cancels the term in $i=4$:

\begin{prop}\label{prop:cancel}
We have $H^i(\Gr, \L^{-k} \oo \Sym \eta)^\chi =0$ for $i=3,4$ and all $k\geq 0$.
\end{prop}

\begin{proof}
Using the sequences (\ref{eq:sym}) tensored by $\L^{-k}$  (and its variants by tensoring with a suitable power of $\bw^3 \Q$) together with Proposition \ref{prop:gr}, we arrive at the connecting morphism $d^{3,4}_k$
\begin{equation}\label{eq:connect}
\begin{aligned}
0 \to H^3(\Gr, \L^{-k} \oo \Sym_{a+1}\eta \oo (\bw^3 \Q)^{b-1})^\chi \to H^3(\Gr, \L^{-k}\oo \Sym_a (\bw^2 \Q \oo \R)\oo(\bw^3 \Q)^{b})^\chi \xrightarrow{\;d^{3,4}_k\;} \\
 \to H^4(\Gr, \L^{-k}\oo\Sym_{a+1}(\bw^2 \Q \oo \R)\oo(\bw^3 \Q)^{b-1})^\chi \to H^4(\Gr, \L^{-k} \oo \Sym_{a+1}\eta\oo(\bw^3 \Q)^{b-1})^\chi \to 0,
\end{aligned}
\end{equation}
where $b=-x-3$ (for $x=-5$ or $x=-8$) and $a=d-b$, with $d,k$ as in Proposition \ref{prop:gr}. With such $k$, we show that $d_k^{3,4}$ is an isomorphism if and only if $d_{k+1}^{3,4}$ is an isomorphism, and the limit morphisms in $\varinjlim_k H^{i} (\Gr, \L^{-k}\oo\Sym \eta)^\chi$ are isomorphisms. Note that $d_k^{3,4}$ is a map between 1-dimensional spaces that come from the third and fourth bundle from (\ref{eq:bundles}). Hence, the inclusion $\L^{-k} \oo \Sym \eta \hookrightarrow \L^{-k-1} \oo \Sym \eta$ induces the vertical maps in the commutative diagram

\[\xymatrix{
H^3(\Gr, \SS_{x,x-1,x-2}\Q \oo \SS_{x+2,x+1,x} \R) \ar[r]^{d_k^{3,4}} \ar[d] & H^4(\Gr, \SS_{x,x-2,x-2}\Q \oo \SS_{x+2,x+2,x} \R) \ar[d] \\
 H^3(\Gr, \SS_{x,x-1,x-2}\Q \oo \SS_{x+2,x+1,x} \R) \ar[r]^{d_{k+1}^{3,4}} &  H^4(\Gr, \SS_{x,x-2,x-2}\Q \oo \SS_{x+2,x+2,x} \R)
}\]

Clearly, the vertical maps are isomorphisms since  $\L^{-k} \oo \Sym \eta \hookrightarrow \L^{-k-1} \oo \Sym \eta$ induces the identity on the level of the respective bundles. Hence, $d_k^{3,4}$ is an isomorphism if and only if $d_{k+1}^{3,4}$ is an isomorphism (of 1-dimensional spaces).

Assume by contradiction that $d_k^{3,4}$ is not an isomorphism for some (hence all) $k$ as in Proposition \ref{prop:gr}. Then $\dim H^3(\Gr, \L^{-k} \oo \Sym \eta)^\chi =1$ by Proposition \ref{prop:gr}, and as seen above, the limit maps in $\varinjlim_k H^{i} (\Gr, \L^{-k}\oo\Sym \eta)^\chi$ are isomorphisms. By Proposition \ref{prop:localopen} we get $\la H^4_{\ol{O_2}}(S/(f)) , \chi \ra=1$, contradicting Proposition \ref{prop:prelim}. Hence, the maps $d_k^{3,4}$ are isomorphisms for all $k$, which by (\ref{eq:connect}) implies that $H^i(\Gr, \L^{-k} \oo \Sym \eta)^\chi =0$ for $i=3,4$.
\end{proof}

We keep the notation as in (\ref{eq:de}) and Proposition \ref{prop:prelim}.

\begin{corollary}\label{cor:ab}
For $k\geq 4$ we have $\dim H^5(\Gr, \L^{-k} \oo \Sym \eta)^{\chi_1} =1$, and for $k\geq 6$ we have $\dim H^8(\Gr, \L^{-k} \oo \Sym \eta)^{\chi_2} =1$. Furthermore, $a_7=1$ and $a_i=0$ for $i\neq 7$, and $b_{10}=1$ and $b_j=0$ for $j\neq 7,8, 10$.
\end{corollary}

\begin{proof}
Clearly, by Proposition \ref{prop:gr} it follows that $\dim H^8(\Gr, \L^{-k} \oo \Sym \eta)^{\chi_2} =1$ for $k\geq 6$. The 3rd and 4th cohomology in Proposition \ref{prop:gr} (1) cancel out in the spectral sequence by Proposition \ref{prop:cancel}. Hence, we have $\dim H^5(\Gr, \L^{-k} \oo \Sym \eta)^{\chi_1} =1$. Analogous to the proof in Proposition \ref{prop:cancel}, the limit maps will be isomorphisms (for $k$ as in Proposition \ref{prop:gr}). By Propositions \ref{prop:localopen} and \ref{prop:prelim}, we obtain $b_{10}=1$ and $a_7=1$, and the rest follows by combining Propositions \ref{prop:localopen}, \ref{prop:prelim}, \ref{prop:gr}, \ref{prop:cancel}. 
\end{proof}

We are left to find $b_7$ and $b_8$. By Propositions \ref{prop:gr} and \ref{prop:cancel}, either $b_7=b_8=0$ or $b_7=b_8=1$. This depends on a connecting homomorphism analogous to (\ref{eq:connect}). Since the limit maps will be isomorphisms (as in the proof of Proposition \ref{prop:cancel}), we can reduce to the case $k=7$ in Proposition \ref{prop:gr}. The connecting homomorphism of interest is
\begin{equation}\label{eq:connect2}
\begin{aligned}
0 \to H^5(\Gr, \L^{-7} \oo \Sym_{3}\eta \oo (\bw^3 \Q)^{2})^{\chi_2} \to H^5(\Gr, \L^{-7}\oo \Sym_2 (\bw^2 \Q \oo \R)\oo(\bw^3 \Q)^{3})^{\chi_2} \xrightarrow{\;d^{5,6}\;} \\
 \to H^6(\Gr, \L^{-7}\oo\Sym_{3}(\bw^2 \Q \oo \R)\oo(\bw^3 \Q)^{2})^{\chi_2} \to H^6(\Gr, \L^{-7} \oo \Sym_{3}\eta\oo(\bw^3 \Q)^{2})^{\chi_2} \to 0.
\end{aligned}
\end{equation}
Among the bundles from (\ref{eq:bundles}), the map $d^{5,6}$ is between
\[H^5(\Gr, \SS_{-9,-9,-11}\Q \oo \SS_{-5,-7,-7} \R)  \xrightarrow{d^{5,6}} H^6(\Gr, \SS_{-9,-10,-11}\Q \oo \SS_{-5,-6,-7} \R).\]

We will obtain the cancellation by comparing the connecting homomorphism (\ref{eq:connect2}) to another one which we know is nonzero. The following finishes the proof of part (2) of the Theorem on Local Cohomology. 


\begin{prop}\label{prop:cancel2}
We have  $H^i(\Gr, \L^{-k} \oo \Sym \eta)^{\chi_2} =0$ for $i=5,6$ and all $k\geq 0$. Hence, $b_7=b_8=0$.
\end{prop}

\begin{proof}
We will track another weight $\sigma=(-6^6)$ in $H^i(\Gr, \L^{-k} \oo \Sym \eta)$ for $i=5,6$. As in the proof of Proposition \ref{prop:gr} (putting $x=-6$), we obtain that $\dim H^5 (\Gr, \L^{-k} \oo \Sym (\gr \eta))^{\sigma} = \dim H^6(\Gr, \L^{-k} \oo \Sym (\gr\eta))^{\sigma}=1$, for all $k\geq 5$. We note that by Theorem \ref{thm:simples}, $D_1^{\sigma} = E^{\sigma} = 0$. However, as the limit maps are isomorphisms (as in the proof of Proposition \ref{prop:cancel}), we get by Proposition \ref{prop:localopen} and (\ref{eq:de}) that we must have $H^5 (\Gr, \L^{-k} \oo \Sym \eta)^{\sigma} = H^6(\Gr, \L^{-k} \oo \Sym \eta)^{\sigma}=0$. This shows that the following connecting homomorphism ${d^{5,6}_\sigma}$ (for $k=5$) for the weight $\sigma$ that is analogous to (\ref{eq:connect2}) is nonzero:
\begin{equation}\label{eq:connect3}
\begin{aligned}
0 \to H^5(\Gr, \L^{-5} \oo \Sym_{3}\eta)^{\sigma} \to H^5(\Gr, \L^{-5}\oo \Sym_2 (\bw^2 \Q \oo \R)\oo\bw^3 \Q)^{\sigma} \xrightarrow{d^{5,6}_\sigma} \\
 \to H^6(\Gr, \L^{-5}\oo\Sym_{3}(\bw^2 \Q \oo \R))^{\sigma} \to H^6(\Gr, \L^{-5} \oo \Sym_{3}\eta)^{\sigma} \to 0.
\end{aligned}
\end{equation}
But this implies that the connecting homomorphism (\ref{eq:connect2}) itself is nonzero, since bundles inducing the two connecting homomorphisms (\ref{eq:connect2}) and (\ref{eq:connect3}) have the same $\O_{\Gr}$-structures -- the difference comes only in the $\GL$-equivariant structure by twisting with the trivial bundle $(\bw^3 Q \oo \bw^3 \R)^{\oo 2}$. Hence $H^i(\Gr, \L^{-k} \oo \Sym \eta)^{\chi_2} =0$ for $i=5,6$.
\end{proof}

\subsection{Local cohomology of the other simple objects}

\begin{theorem}\label{thm:locD1}
For all orbits $O\neq O_0$ we have $H^{\bullet}_{\ol{O}}(D_1)=H^0_{\ol{O}}(D_1)=D_1$. The following are all the nonzero local cohomology modules of $D_1$ with support in $O_0$:
$$
H^4_{O_0}(D_1)=E,\;\; H^6_{O_0}(D_1)=E,\;\; H^{10}_{O_0}(D_1)=E.
$$
\end{theorem}

\begin{proof}
The first assertion follows immediately from the fact that the support of $D_1$ is $\ol{O_1}$. Now we prove the second assertion. Using parts (0) and (1) of the Theorem on Local Cohomology, the spectral sequence $H^i_{O_0}(H^j_{\ol{O_1}}(S))\Rightarrow H^{i+j}_{O_0}(S)$ does not converge properly unless $H^4_{O_0}(D_1)=E$, $H^6_{O_0}(D_1)=E$, $H^{10}_{O_0}(D_1)=E$, and all other local cohomology modules vanish.
\end{proof}

\begin{lemma}\label{lem:H7B4}
We have $H^7_{O_0}(B_4)=0$, and for $i\geq 0$ $H^{i+1}_{O_0}(B_4)=H^i_{O_0}(H^5_{\ol{O_2}}(S))$.
\end{lemma}

\begin{proof}
By (\ref{eq:HB4}), we have $H^1_{\ol{O_3}} (B_4) \cong H^5_{\ol{O_2}}(S)$, and the spectral sequence $H^i_{O_0} (H^j_{\ol{O_3}}(B_4)) \Rightarrow H^{i+j}_{O_0} (B_4)$ degenerates. 

We are left to show that $H^7_{O_0}(B_4)=0$. Since $E$ is the injective hull of the $S$-module $\C \cong S/\mf{m}$ (where $\mf{m}$ denotes the homogeneous maximal ideal in $S$), it enough to show $\Hom_S(\C, H^7_{O_0}(B_4))=0$, or equivalently $\Hom_S(\C \oo (-10^6), H^7_{O_0}(B_4))^G=0$ by taking into account equivariance. Since $H^i_{O_0}(B_4)$ is an injective $S$-module for all $i\geq 0$, the spectral sequence $\Ext^i_S(\C, H^j_{O_0}(B_4)) \Rightarrow \Ext^{i+j}_S(\C,B_4)$ yields $\Hom_S(\C \oo (-10^6), H^7_{O_0}(B_4))^G \cong \Ext^7_S(\C \oo (-10^6), B_4)^G$. Using the Koszul resolution for $\C$, it is enough to show that $B_4 \oo \bw^7  (\bw^3 W^*) \oo (10^6)$ has no $G$-invariant sections, or equivalently none of the irreducible $G$-modules in $\bw^7  (\bw^3 W) \oo (-10^6)$ appear in $B_4$. Using the computer algebra system Magma \cite{magma}, we obtain the plethysm
\[
\begin{aligned}[rl]
[\bw^7 \bw^3 W] &= e^{(4,4,4,3,3,3)} + e^{(5,4,4,4,2,2)}+e^{(5,5,3,3,3,2)}+e^{(5,5,4,4,2,1)}+e^{(5,5,5,3,3,0)}+e^{(6,4,4,3,3,1)}+\\
&+e^{(6,4,4,4,3,0)}+e^{(6,5,3,3,2,2)}+e^{(6,5,4,3,2,1)}+e^{(7,4,3,3,3,1)} + e^{(7,4,4,2,2,2)}.
\end{aligned}
\]
By inspection, we see using Theorem \ref{thm:charB4} that none of the irreducibles in $\bw^7  (\bw^3 W) \oo (-10^6)$ appear in $[B_4]$.
\end{proof}

\begin{theorem}\label{thm:locD2}
For all orbits $O\neq O_0,O_1$ we have $H^{\bullet}_{\ol{O}}(D_2)=H^0_{\ol{O}}(D_2)=D_2$. All of the nonzero local cohomology of $D_2$ with support in $O_0$ and $\ol{O_1}$ is described as follows:
\begin{itemize}
\item[(0)] $H^5_{O_0}(D_2)=E$,  $H^7_{O_0}(D_2)=E$, $H^9_{O_0}(D_2)=E$, $H^{11}_{O_0}(D_2)=E$, $H^{13}_{O_0}(D_2)=E$, $H^{15}_{O_0}(D_2)=E$,\smallskip

\item[(1)] $H^1_{\ol{O_1}}(D_2)=D_1$, $H^3_{\ol{O_1}}(D_2)=D_1$, $H^5_{\ol{O_1}}(D_2)=D_1$, $H^6_{\ol{O_1}}(D_2)=E$, $H^8_{\ol{O_1}}(D_2)=E$, $H^{10}_{\ol{O_1}}(D_2)=E$. 
\end{itemize}
\end{theorem}

\begin{proof}
To prove (1), we consider the spectral sequence $H^i_{\ol{O_1}}(H^j_{\ol{O_2}}(S))\Rightarrow H^{i+j}_{\ol{O_1}}(S)$. Using (1) and (2) in the Theorem on Local Cohomology and Theorem \ref{thm:locD1}, the spectral sequence does not converge properly unless the local cohomology modules are as claimed.

To prove (0), we consider the spectral sequence $H^i_{O_0}(H^j_{\ol{O_2}}(S))\Rightarrow H^{i+j}_{O_0}(S)$. Note that by Lemma \ref{lem:H7B4} we have $H^6_{O_0}(H^5_{\ol{O_2}}(S)) = 0$. Using this, together with parts (0) and (2) in the Theorem on Local Cohomology and Theorem \ref{thm:locD1}, the spectral sequence does not converge properly unless the local cohomology modules are as claimed.
\end{proof}

\begin{lemma}\label{lem:fOne}
For all $O\neq O_4$, we have $H^{\bullet}_{\ol{O}}(S_f)=0$, $H^{\bullet}_{\ol{O}}(S_f\cdot \sqrt{f})=0$, and $H^{\bullet}_{\ol{O}}(\mathcal{D}f^{-1})=H^1_{\ol{O}}(\mathcal{D}f^{-1})=E$.
\end{lemma}

\begin{proof}
The proofs of the assertions about $S_f$ and $S_f\cdot \sqrt{f}$ are analogous to the proof of \cite[Lemma 6.9]{lHorincz2018iterated}, replacing $\tn{det}$ with $f$. To prove the final assertion, consider the short exact sequence
$$
0\longrightarrow \mathcal{D}f^{-1}\longrightarrow S_f \longrightarrow E\longrightarrow 0.
$$
Since $H^{\bullet}_{\ol{O}}(E)=H^0_{\ol{O}}(E)=E$ for all orbits $O$, the long exact sequence of local cohomology and the first assertion yield the desired result.
\end{proof}

\begin{theorem}\label{thm:locD3}
The following are all the nonzero local cohomology modules of $D_3$ with support in an orbit closure:
\begin{itemize}
\item[(0)] $H^{1}_{O_0}(D_3) = E$, $H^{19}_{O_0}(D_3) = E$. \smallskip
\item[(1)] $H^{1}_{\ol{O_1}}(D_3) = E,\;\;H^{9}_{\ol{O_1}}(D_3) = D_1, \;\; H^{12}_{\ol{O_1}}(D_3) = E, \;\; H^{14}_{\ol{O_1}}(D_3) = E$.\smallskip
\item[(2)] $H^{1}_{\ol{O_2}}(D_3) = E,\;\; 0\to D_2 \to H^{4}_{\ol{O_2}}(D_3) \to D_1 \to 0, \; \; H^{6}_{\ol{O_2}}(D_3) = D_1, \; \; H^{9}_{\ol{O_2}}(D_3) = E.$\smallskip
\item[(3)] $H^0_{\ol{O_3}}(D_3)=D_3.$
\end{itemize}
\end{theorem}

\begin{proof}
Consider the following short exact sequence 
$$
0\longrightarrow S \longrightarrow \mathcal{D}f^{-1}\longrightarrow D_3 \longrightarrow 0.
$$
Using the Theorem on Local Cohomology and Lemma \ref{lem:fOne}, the result follows by the long exact sequence of local cohomology.
\end{proof}

\begin{lemma}\label{lem:Five}
For all orbits $O\neq O_0,O_4$, we have $H^{\bullet}_{\ol{O}}(\mathcal{D}f^{-5/2})=H^1_{\ol{O}}(\mathcal{D}f^{-5/2})=D_1$. All of the nonzero local cohomology of $\mathcal{D}f^{-5/2}$ with support in $O_0$ is described as follows:
$$
H^5_{O_0}(\mathcal{D}f^{-5/2})=E,\;\;H^7_{O_0}(\mathcal{D}f^{-5/2})=E,\;\; H^{11}_{O_0}(\mathcal{D}f^{-5/2})=E.
$$
\end{lemma}

\begin{proof}
All assertions follow from Theorem \ref{thm:locD1} and Lemma \ref{lem:fOne}, using the long exact sequence of local cohomology coming from the inclusion $\mathcal{D}f^{-5/2}\subset S_f\cdot \sqrt{f}$.
\end{proof}

\begin{theorem}\label{thm:locB4}
The following are all the nonzero local cohomology modules of $B_4$ with support in an orbit closure:
\begin{itemize}
\item[(0)] $H^{10}_{O_0}(B_4)=E,\;\; H^{14}_{O_0}(B_4)=E,\;\; H^{16}_{O_0}(B_4)=E,$\smallskip
\item[(1)] $H^4_{\ol{O_1}}(B_4)=D_1,\;\; H^6_{\ol{O_1}}(B_4)=D_1,\;\; H^7_{\ol{O_1}}(B_4)=E,\;\; H^9_{\ol{O_1}}(B_4)=E,\;\; H^{11}_{\ol{O_1}}(B_4)=E.$\smallskip
\item[(2)] $0\to D_2 \to H^1_{\ol{O_2}}(B_4)\to D_1\to 0.$\smallskip
\item[(3)] $0\to D_2 \to H^1_{\ol{O_3}}(B_4)\to D_1\to 0.$
\end{itemize}
\end{theorem}

\begin{proof}
Part (0) follows by Lemma \ref{lem:H7B4} since the modules $H^i_{O_0}(H^5_{\ol{O_2}}(S))$ were computed in the $E_2$-page of the spectral sequence used in the proof of Theorem \ref{thm:locD2} part (0).

To prove (1), consider the short exact sequence 
\begin{equation}\label{sesB4}
0\longrightarrow B_4 \longrightarrow \mathcal{D}f^{-5/2}\longrightarrow D_2\longrightarrow 0.
\end{equation}
By Theorem \ref{thm:locD2} and Lemma \ref{lem:Five}, we obtain an exact sequence
\begin{equation}
0\longrightarrow H^1_{\ol{O_1}}(B_4)\longrightarrow D_1 \longrightarrow D_1 \longrightarrow H^2_{\ol{O_1}}(B_4)\longrightarrow 0,
\end{equation}
and $H^i_{\ol{O_1}}(B_4)\cong H^{i-1}_{\ol{O_1}}(D_2)$ for $i\geq 3$. Therefore, it suffices to show that $H^1_{\ol{O_1}}(B_4)=0$. Let $U=X\setminus \ol{O_1}$ and let $j:U\hookrightarrow X$ be the corresponding open immersion. We have the exact sequence
\begin{equation}\label{eq:H1B4}
0 \to B_4 \lra j_*j^*B_4 \lra H^1_{\ol{O_1}}(B_4) \to 0.
\end{equation}
By adjointness, we see that $\End_\D ( j_*j^*B_4) \cong \C$, so $j_*j^*B_4$ is indecomposable. Since $H^1_{\ol{O_1}}(B_4)$ must be some direct sum of the $\D$-modules $E$ and $D_1$, the Theorem on the Quiver Structure implies that the sequence (\ref{eq:H1B4}) must split. Since $j_*j^*B_4$ is indecomposable, we must have $H^1_{\ol{O_1}}(B_4)=0$.


Part (3) follows from as explained in (\ref{eq:HB4}), and part (2) follows immediately from (3), using the spectral sequence $H^i_{\ol{O_2}}(H^j_{\ol{O_3}}(B_4))\Rightarrow H^{i+j}_{\ol{O_2}}(B_4)$.
\end{proof}

We obtained all the local cohomology modules of each simple equivariant $\D$-module. Note that the indecomposable $\D$-modules $H^1_{\ol{O_3}} (S)$ and $H^1_{\ol{O_3}}(B_4)$ are the only local cohomology modules that are not simple among these. Moreover, the local cohomology of these indecomposables can also be obtained from $H^i_{\ol{O}}(H^1_{\ol{O_3}}(S)) = H^{i+1}_{\ol{O}}(S)$ and $H^i_{\ol{O}}(H^1_{\ol{O_3}}(B_4))= H^{i+1}_{\ol{O}}(B_4)$ (where $O\neq O_4$), since the respective spectral sequences degenerate. This shows that we can compute all \textit{iterated} local cohomology modules of each simple equivariant $\D$-module. In particular, computing $H^i_{O_0}(H^j_{\ol{O}}(S))$ we obtain the Lyubeznik numbers $\ll_{i,j}(\C[\ol{O}]_{\mf{m}})$ (see \cite{lyubeznik}, or the survey \cite{lyub-survey}) for each orbit closure $\ol{O}$, where $\C[\ol{O}]_{\mf{m}}$ is the localization of the coordinate ring of $\ol{O}$ at the maximal homogeneous ideal $\mf{m}$. Since $\ol{O_3}$ is a hypersurface, its only nonzero Lyubeznik number is $\ll_{19,19}(\C[\ol{O_3}]_{\mf{m}})=1$ (see also \cite[Section 4]{lyub-survey}). The nontrivial cases are $\ol{O_1}$ and $\ol{O_2}$:

\begin{corollary}\label{cor:lyubez}
For the rings $R_1 = \C[\ol{O_1}]_{\mf{m}}$ and $R_2=\C[\ol{O_2}]_{\mf{m}}$, the following are the only nonzero Lyubeznik numbers (in which case they are equal to 1):
\begin{itemize}
\item[(1)] $\ll_{0,5}(R_1), \; \ll_{0,7}(R_1),\; \ll_{4,10}(R_1), \; \ll_{6,10}(R_1), \; \ll_{10,10}(R_1);$
\item[(2)] $\ll_{0,10}(R_2), \; \ll_{4,13}(R_2), \; \ll_{6,13}(R_2), \; \ll_{10,13}(R_2), \; \ll_{9,15}(R_2), \; \ll_{13,15}(R_2), \; \ll_{15,15}(R_2).$
\end{itemize}

\end{corollary}

\section*{Acknowledgments}

We are grateful to Claudiu Raicu and Jerzy Weyman for helpful conversations and suggestions. Perlman acknowledges the support of the NSF Graduate Research Fellowship under Grant No. DGE-1313583.

\bibliographystyle{alpha}
\bibliography{mybib}

\begin{bibdiv}
\begin{biblist}

\bib{ass-sim-sko}{book}{
      author={Assem, Ibrahim},
      author={Simson, Daniel},
      author={Skowro\'nski, Andrzej},
       title={Elements of the representation theory of associative algebras.
  {V}ol. 1},
      series={London Mathematical Society Student Texts},
   publisher={Cambridge University Press, Cambridge},
        date={2006},
      volume={65},
}

\bib{magma}{article}{
      author={Bosma, Wieb},
      author={Cannon, John},
      author={Playoust, Catherine},
       title={The {M}agma algebra system. {I}. {T}he user language},
        date={1997},
     journal={J. Symbolic Comput.},
      volume={24},
      number={3-4},
       pages={235\ndash 265},
        note={Computational algebra and number theory (London, 1993)},
}

\bib{brion}{article}{
      author={Brion, Michel},
       title={Invariants d'un sous-groupe unipotent maximal d'un groupe
  semi-simple},
        date={1983},
     journal={Ann. Inst. Fourier (Grenoble)},
      volume={33},
      number={1},
       pages={1\ndash 27},
}

\bib{garsab}{article}{
      author={Garc\'{\i}a~L\'{o}pez, R.},
      author={Sabbah, C.},
       title={Topological computation of local cohomology multiplicities},
        date={1998},
        ISSN={0010-0757},
     journal={Collect. Math.},
      volume={49},
      number={2-3},
       pages={317\ndash 324},
        note={Dedicated to the memory of Fernando Serrano},
      review={\MR{1677136}},
}

\bib{hartshorne}{book}{
      author={Hartshorne, Robin},
       title={Algebraic geometry},
   publisher={Springer-Verlag, New York-Heidelberg},
        date={1977},
        note={Graduate Texts in Mathematics, No. 52},
}

\bib{hot-tak-tan}{book}{
      author={Hotta, Ryoshi},
      author={Takeuchi, Kiyoshi},
      author={Tanisaki, Toshiyuki},
       title={{$D$}-modules, perverse sheaves, and representation theory},
      series={Progress in Mathematics},
   publisher={Birkh\"auser Boston, Inc., Boston, MA},
        date={2008},
      volume={236},
        note={Translated from the 1995 Japanese edition by Takeuchi},
}

\bib{igusa}{book}{
      author={Igusa, Jun-ichi},
       title={An introduction to the theory of local zeta functions},
      series={AMS/IP Studies in Advanced Mathematics},
   publisher={American Mathematical Society, Providence, RI; International
  Press, Cambridge, MA},
        date={2000},
      volume={14},
}

\bib{kashi}{book}{
      author={Kashiwara, Masaki},
       title={{$D$}-modules and microlocal calculus},
      series={Iwanami Series in Modern Mathematics; Translations of
  Mathematical Monographs},
   publisher={AMS},
     address={Providence, RI},
        date={2003},
      volume={217},
}

\bib{kimu}{article}{
      author={Kimura, Tatsuo},
       title={{T}he b-functions and holonomy diagrams of irreducible regular
  prehomogeneous vector spaces},
        date={1982},
     journal={Nagoya Math. J.},
      volume={85},
       pages={1\ndash 80},
}

\bib{knopmenzel}{article}{
      author={Knop, Friedrich},
      author={Menzel, Gisela},
       title={Duale {V}ariet\"aten von {F}ahnenvariet\"aten},
        date={1987},
     journal={Comment. Math. Helv.},
      volume={62},
      number={1},
       pages={38\ndash 61},
}

\bib{bub}{article}{
      author={L\H{o}rincz, Andr{\'a}s~C},
       title={Decompositions of {B}ernstein-{S}ato polynomials and slices},
        date={2019},
     journal={Transform. Groups},
        note={DOI: 10.1007/s00031-019-09526-7},
}

\bib{series}{article}{
      author={Landsberg, Joseph~M},
      author={Manivel, Laurent},
       title={Series of {L}ie groups},
        date={2004},
     journal={Michigan Math. J.},
      volume={52},
      number={2},
       pages={453\ndash 479},
}

\bib{lHorincz2018iterated}{article}{
      author={L{\H{o}}rincz, Andr{\'a}s~C},
      author={Raicu, Claudiu},
       title={Iterated local cohomology groups and {L}yubeznik numbers for
  determinantal rings},
        date={2018},
     journal={arXiv},
      number={1805.08895},
}

\bib{bindmod}{article}{
      author={L{\H{o}}rincz, Andr{\'a}s~C},
      author={Raicu, Claudiu},
      author={Weyman, Jerzy},
       title={Equivariant {$\mathcal{D}$}-modules on binary cubic forms},
        date={2019},
     journal={Comm. Algebra},
      volume={47},
       pages={2457\ndash 2487},
}

\bib{LSW}{article}{
      author={Lyubeznik, Gennady},
      author={Singh, Anurag~K.},
      author={Walther, Uli},
       title={Local cohomology modules supported at determinantal ideals},
        date={2016},
        ISSN={1435-9855},
     journal={J. Eur. Math. Soc. (JEMS)},
      volume={18},
      number={11},
       pages={2545\ndash 2578},
         url={https://doi.org/10.4171/JEMS/648},
      review={\MR{3562351}},
}

\bib{catdmod}{article}{
      author={L{\H{o}}rincz, Andr{\'a}s~C},
      author={Walther, Uli},
       title={On categories of equivariant {$D$}-modules},
        date={2019},
     journal={Adv. Math.},
      volume={351},
       pages={429\ndash 478},
}

\bib{lyubeznik}{article}{
      author={Lyubeznik, Gennady},
       title={Finiteness properties of local cohomology modules (an application
  of {$D$}-modules to commutative algebra)},
        date={1993},
     journal={Invent. Math.},
      volume={113},
      number={1},
       pages={41\ndash 55},
}

\bib{manivel}{book}{
      author={Manivel, Laurent},
       title={Symmetric functions, {S}chubert polynomials and degeneracy loci},
      series={SMF/AMS Texts and Monographs},
   publisher={American Mathematical Society, Providence, RI; Soci\'{e}t\'{e}
  Math\'{e}matique de France, Paris},
        date={2001},
      volume={6},
        ISBN={0-8218-2154-7},
        note={Translated from the 1998 French original by John R. Swallow,
  Cours Sp\'{e}cialis\'{e}s [Specialized Courses], 3},
      review={\MR{1852463}},
}

\bib{mac-vil}{article}{
      author={MacPherson, Robert},
      author={Vilonen, Kari},
       title={Elementary construction of perverse sheaves},
        date={1986},
     journal={Invent. Math.},
      volume={84},
      number={2},
       pages={403\ndash 435},
}

\bib{macvil}{article}{
      author={Mac{P}herson, Robert},
      author={Vilonen, Kari},
       title={Elementary construction of perverse sheaves},
        date={1986},
     journal={Invent. Math.},
      volume={84},
       pages={403\ndash 435},
}

\bib{lyub-survey}{article}{
      author={N\'u\~nez Betancourt, Luis},
      author={Witt, Emily~E.},
      author={Zhang, Wenliang},
       title={A survey on the {L}yubeznik numbers},
        date={2016},
       pages={137\ndash 163},
}

\bib{Ogus}{article}{
      author={Ogus, Arthur},
       title={Local cohomological dimension of algebraic varieties},
        date={1973},
        ISSN={0003-486X},
     journal={Ann. of Math. (2)},
      volume={98},
       pages={327\ndash 365},
         url={https://doi.org/10.2307/1970785},
      review={\MR{506248}},
}

\bib{mike}{article}{
      author={Perlman, Michael},
       title={Equivariant {$\mathcal{D}$}-modules on {$2\times 2\times 2$}
  hypermatrices},
        date={2018},
     journal={arXiv},
      number={1809.00352},
}

\bib{claudiu1}{article}{
      author={Raicu, Claudiu},
       title={Characters of equivariant {$\mathcal{D}$}-modules on spaces of
  matrices},
        date={2016},
     journal={Compos. Math.},
      volume={152},
       pages={1935\ndash 1965},
}

\bib{claudiu5}{article}{
      author={Raicu, Claudiu},
       title={Characters of equivariant {$\mathcal{D}$}-modules on {V}eronese
  cones},
        date={2017},
     journal={Trans. Amer. Math. Soc.},
      volume={369},
      number={3},
       pages={2087\ndash 2108},
}

\bib{claudiu2}{article}{
      author={Raicu, Claudiu},
      author={Weyman, Jerzy},
       title={Local cohomology with support in generic determinantal ideals},
        date={2014},
     journal={Algebra Number Theory},
      volume={8},
      number={5},
       pages={1231\ndash 1257},
}

\bib{claudiu4}{article}{
      author={Raicu, Claudiu},
      author={Weyman, Jerzy},
       title={Local cohomology with support in ideals of symmetric minors and
  {P}faffians},
        date={2016},
     journal={J. Lond. Math. Soc. (2)},
      volume={94},
      number={3},
       pages={709\ndash 725},
}

\bib{claudiu3}{article}{
      author={Raicu, Claudiu},
      author={Weyman, Jerzy},
      author={Witt, Emily},
       title={Local cohomology with support in ideals of maximal minors and
  sub-maximal {P}faffians},
        date={2014},
     journal={Adv. Math.},
      volume={250},
       pages={596\ndash 610},
}

\bib{saito}{article}{
      author={Saito, Morihiko},
       title={On b-function, spectrum and rational singularity},
        date={1993},
     journal={Math. Ann.},
      volume={295},
       pages={51\ndash 74},
}

\bib{saki}{article}{
      author={Sato, Mikio},
      author={Kimura, Tatsuo},
       title={A classification of irreducible prehomogeneous vector spaces and
  their relative invariants},
        date={1977},
     journal={Nagoya Mathematical Journal},
      volume={65},
       pages={1\ndash 155},
}

\bib{switala2015lyubeznik}{article}{
      author={Switala, Nicholas},
       title={Lyubeznik numbers for nonsingular projective varieties},
        date={2015},
     journal={Bulletin of the London Mathematical Society},
      volume={47},
      number={1},
       pages={1\ndash 6},
}

\bib{vilonen}{article}{
      author={Vilonen, Kari},
       title={Perverse sheaves and finite-dimensional algebras},
        date={1994},
     journal={Trans. Amer. Math. Soc.},
      volume={341},
      number={2},
       pages={665\ndash 676},
}

\bib{jerzy}{book}{
      author={Weyman, Jerzy},
       title={Cohomology of vector bundles and syzygies},
      series={Cambridge Tracts in Mathematics},
   publisher={Cambridge University Press},
     address={Cambridge},
        date={2003},
      volume={149},
}

\bib{jerzydisc}{article}{
      author={Weyman, Jerzy},
       title={Calculating discriminants by higher direct images},
        date={1994},
     journal={Trans. Amer. Math. Soc.},
      volume={343},
      number={1},
       pages={367\ndash 389},
}

\end{biblist}
\end{bibdiv}

\end{document}